\documentclass[12pt]{article}
\usepackage{amsmath, latexsym, amsfonts, amssymb, amsthm, amscd,bbm}
\usepackage[english]{babel}
\usepackage[T1]{fontenc}
\usepackage{caption}
\usepackage{subcaption}

\usepackage{graphicx}
\def \R{\mathbb R}
\def \N{\mathbb N}

\def \E{\mathbb E}

\usepackage{bbm}

\usepackage{color}

\newtheorem{thm}{Theorem}[section]
\newtheorem{cor}[thm]{Corollary}
\newtheorem{lem}[thm]{Lemma}
\newtheorem{prop}[thm]{Proposition}

\newtheorem{rem}[thm]{Remark}

\newtheorem{lemma}[thm]{Lemma}

\newtheorem{assump}[thm]{Assumption}

\DeclareMathOperator{\bE}{\mathbf{E}}

\newcommand{\norm}[1]{\left \| #1 \right \|}

\renewcommand{\tilde}{\widetilde}

\newcommand{\cK}{{\ensuremath{\mathcal K}} }

\newcommand{\cL}{{\ensuremath{\mathcal L}} }

\newcommand{\one}{{\ensuremath{\mathbbm{1}}}}

\numberwithin{equation}{section}

\usepackage[nottoc]{tocbibind}
\usepackage[colorlinks=true, citecolor=blue]{hyperref}
\usepackage{makeidx}
\makeindex

\author{J. Inglis and J. MacLaurin\\
\small{Inria Sophia Antipolis}}

\title{A general framework for stochastic traveling waves and patterns, with application to neural field equations\footnote{
This work was partially supported by the European Union Seventh Framework Programme (FP7) under grant agreement no. 269921 (BrainScaleS), no. 318723 (Mathemacs) and the Human Brain Project (HBP).
}}

\begin{document}
\maketitle

\begin{abstract}
In this paper we present a general framework in which to rigorously study the effect of spatio-temporal noise on traveling waves and stationary patterns. In particular the framework can incorporate versions of the stochastic neural field equation that may exhibit traveling fronts, pulses or stationary patterns.  To do this, we first formulate a local SDE that describes the position of the stochastic wave up until a discontinuity time, at which point the position of the wave may jump.  We then study the local stability of this stochastic front, obtaining a result that recovers a well-known deterministic result in the small-noise limit.  We finish with a study of the long-time behavior of the stochastic wave.
\end{abstract}

\section{Introduction}
Deterministic traveling waves have been widely used to model phenomena in a huge range of scientific areas, including chemical kinetics, population dynamics, combustion, transport in porous media, electroconvection and neuroscience. More generally, equations that exhibit spatial patterns are ubiquitous in the biomedical sciences and are a key lens through which emergent phenomena are studied (see for example \cite{maini1997spatial, murray2001mathematical, volpert-volpert} and \cite{xin}).  However, the effect of noise on these equations is much less well-developed, and works in this direction have in the past tended to focus either on specific situations (see for example \cite{brassesco-demassi-presutti, funaki} for the case of the Ginzburg-Landau equation or \cite{doering-mueller-smereka, harris-harris-kyprianou} for the FKPP equation), or numerical approximations (see for example \cite{lord-thummler}).

The goal of this paper is to introduce a framework in which it is possible to study stochastic perturbations of traveling wave solutions 
to a general class of evolution equations (which may include PDEs and integral equations). Our specific motivation is the recent interest in stochastic versions of the neural field equation (\cite{bressloff-kilpatrick, bressloff-webber, inglis-faugeras, kruger-stannat}).  The (deterministic) neural field equation and its variants are used in the neuroscience literature to model the spatio-temporal dynamics of macroscopic cortical activity (see \cite{bressloff-review} for a review).  In particular, as outlined in more detail in Section \ref{examples} below, one reason these equations are interesting is that they exhibit a traveling wave solution of the form $u(t, x) = \varphi_0(x-ct)$ for all $t\geq0$, $x\in\R$ and some speed $c\in\R$, where the wave form $\varphi_0$ satisfies the stationary equation
\begin{equation}
\label{into:stationary}
0 = A\varphi_0 + f(\varphi_0),
\end{equation}
and $A$ and $f$ are explicit linear and nonlinear operators respectively.  Due to translation invariance, it follows that $\varphi_\alpha:= \varphi_0(\cdot + \alpha)$ is also a solution for any $\alpha\in\R$, so that we in fact have a family $(\varphi_\alpha)_{\alpha\in\R}$ of solutions to \eqref{into:stationary}.  The stochastic evolution equation of interest is then given by
\begin{equation}
\label{intro:sde}
du_t = [Au_t + f(u_t)]dt + \varepsilon B(t)dW^Q_t, \quad t\geq0,
\end{equation}
whose solution $(u_t)_{t\geq0}$ is a functional-valued process i.e. $u_t:\R \to\R$ for all $t\geq0$. 
Here $\varepsilon>0$, $(W^Q_t)_{t\geq0}$ is a Hilbert space-valued noise and $B(t)$ is an operator-valued diffusion coefficient made precise below.  However, instead of working in the specific case of these neural field equations, we instead formulate general conditions on $A$, $f$ and $(\varphi_\alpha)_{\alpha\in\R}$ that allow us to study the effect of noise on a general class of wave and pattern forms.  The conditions are broad enough to include the important cases of traveling fronts and pulses.

One of the main ideas used in our work (developing those presented in \cite{bressloff-webber} and \cite{kruger-stannat}), is to compare the solution $(u_t)_{t\geq0}$ of \eqref{intro:sde} to the family of deterministic fronts $(\varphi_\alpha)_{\alpha\in\R}$. It is clear that if $\varepsilon =0$ and $u_0 = \varphi_0$ then $u_t = \varphi_0$ for all $t\geq0$.  However, when $\varepsilon>0$ the `stochastic front' will move in time i.e. the noise will influence the speed of the wave.  To describe this movement, it is natural to consider the dynamics of the global minimum of the map
\begin{equation}
\label{intro:min}
\alpha \mapsto \|u_t - \varphi_\alpha\|^2, \quad t\geq0,
\end{equation}
where $\|\cdot\|$ is the norm on an appropriate Hilbert space.
Indeed, if $\alpha$ attains this minimum, then $\varphi_\alpha$ is the front closest to $u_t$, and we say that the stochastic front is at position $\alpha\in\R$.  However, a key point our analysis highlights is that the dynamics of a global minimum of  \eqref{intro:min} may be quite complicated.  In particular the global minimum may not be uniquely defined, may be discontinuous as a function of time, and there may exist many local minima (meaning that a gradient-descent method to approximate the minimum of \eqref{intro:min} may only converge towards one of many local minimum).

Despite these complications, in Section \ref{sec:Tracking the wave front} below, we show that we can locally describe the behavior of any local minimum of \eqref{intro:min} with an SDE.  This goes further than the work of \cite{bressloff-webber} and \cite{kruger-stannat}, since our description is exact rather than a first order $\varepsilon$-expansion or an approximation.  We can also see that the solution of the SDE exists exactly up until the point at which the local minimum may become a saddle point.

The second part of this work (Sections \ref{sec:local stability} and \ref{sec:long-time stability}) focuses on the local stability for small $\varepsilon$ and long-time behavior of the stochastic wave fronts.  
An important result from the deterministic literature on traveling waves is that under some conditions (in particular on the spectrum of $A$) and in the case when $\varepsilon = 0$, if the initial condition $\|u_0 - \varphi_0\|$ is small enough, then there exists an $\alpha\in\R$ such that
\[
\|u_t - \varphi_\alpha\| \leq Me^{-bt}, \quad t\geq0,
\]
for some constants $M>0$ and $b>0$ i.e. the solution to \eqref{intro:sde} converges exponentially fast to one of the deterministic fronts.  A natural question is therefore to ask if there exist related results in the stochastic setting, where one can recover the deterministic result in the limit as $\varepsilon \to 0$.  One of our main results (Corollary \ref{cor:ineq}) does exactly this.  It is worth highlighting that our techniques do not involve any order expansions in $\varepsilon$. The drawback of this result is that it is local in nature, since it guarantees convergence only up until the first time that the noise becomes too big (although of course this becomes infinite in the limit as $\varepsilon\to0$).  The aim of the final section (Section \ref{sec:long-time stability}) is thus to try and study the long-time behavior of $\|u_t - \varphi_{\beta^*_t}\|^2$, where $\beta_t^*$ is any global minimum of the map \eqref{intro:min} for all $t\geq0$.  As mentioned above, this analysis is complicated by the fact that the process $\beta_t^*$ is highly discontinuous.  However, we can still derive a description of $\|u_t - \varphi_{\beta^*_t}\|^2$ for all $t\geq0$ under some conditions (see Theorem \ref{Proposition Nonlinear Inequality}).

The organization of the paper is as follows.  In Section \ref{General Setting} we describe the general deterministic setting we consider, and state our assumptions.   Section \ref{examples} then goes on to describe three motivating examples that fit into the general setting.  Section \ref{sec:stoch wave} introduces the stochastic version of the general traveling wave equation, and shows that such equations are well-posed, while in Section \ref{sec:Tracking the wave front} we describe what we mean by the position of the stochastic front.  Finally, as mentioned, Sections \ref{sec:local stability} and \ref{sec:long-time stability} deal with the local stability and long-time behavior of the stochastic wave fronts respectively.

\vspace{0.3cm}
\noindent\textit{Notation:}
As usual, $\mathcal{C}(\R^d)$ and $\mathcal{C}^\infty(\R^d)$ will denote the spaces of real-valued functions on $\R^d$ that are continuous and smooth respectively.  Moreover $L^p(\R^d)$ ($p\geq1$), will be the space of $p$-integrable functions with respect to the Lebesgue measure on $\R^d$.  Finally, for general Banach spaces $E_1, E_2$, we will denote by $L(E_1, E_2)$ the space of bounded linear operators $:E_1 \to E_2$.

\section{General setting}
\label{General Setting}

Let $E_0$ be a Banach space of $\R^N$-valued functions over $\R^d$, for $N,d\geq1$. Let $A$ and $f$ be linear and nonlinear operators respectively acting in $E_0$.
Suppose that there exists a family
$(\varphi_\alpha)_{\alpha\in\R} \subset E_0$ such that
\begin{equation}
\label{stationary equation}
A\varphi_\alpha+ f(\varphi_\alpha) = 0,\ \forall \alpha\in\R.
\end{equation}
Let $H:=[L^2(\R^d)]^N$, equipped with the standard inner product denoted by $\langle\cdot, \cdot \rangle$ and norm $\|\cdot\|$.
Let $E:= \varphi_0 + H$ (i.e. $u\in E$ if and only if $u = \varphi_0 + v$ for some $v$ in $H$), endowed with the topology inherited from $H$.

We make use of the following assumptions on $(\varphi_\alpha)_{\alpha\in\R}$, $f$ and $A$, which are similar to those imposed in \cite[Chapter 5]{volpert-volpert}.

\begin{assump}
\label{assump varphi}
Assume that the family $(\varphi_\alpha)_{\alpha\in\R}$ satisfies the following conditions.
\begin{itemize}
\item [(i)] The derivatives $[d^k/d\alpha^k] \varphi_\alpha$ (the derivatives being taken in the norm of the space $H$) exist for $k\in\{1, 2, 3\}$ 
and are all in the space $H$.  We will denote these derivatives by $\varphi'_\alpha, \varphi''_\alpha$, and $\varphi'''_\alpha$ respectively.  
\item[(ii)] $\alpha\mapsto \varphi'_\alpha, \varphi''_\alpha, \varphi'''_\alpha$ are all globally Lipschitz, $\|\varphi'_\alpha\|, \|\varphi''_\alpha\|, \|\varphi'''_\alpha\|$ are all independent of $\alpha$, and integration by parts holds i.e. $\langle \varphi_0', \varphi'_0\rangle = -\langle  \varphi''_0, \varphi_0\rangle$.
\item[(iii)] $\varphi'_\alpha\in\mathcal{D}(A^*)$ for all $\alpha\in\R$, $\alpha \mapsto A^*\varphi'_\alpha$ is globally Lipschitz and $\|A^*\varphi'_\alpha\|$ is independent of $\alpha$.
\item[(iv)] $\langle\varphi'_\alpha,\varphi'_{\beta+\alpha}\rangle \to 0$ as $|\beta| \to \infty$, uniformly in $\alpha\in\R$, $\lim_{|\alpha|\to\infty}|\langle \varphi_\alpha', \varphi_\alpha - \varphi_0\rangle|>0$ and either of the following hold:
\begin{itemize}
\item[(a)] $\|\varphi_\alpha - \varphi_0\| \to \infty$ as $|\alpha| \to \infty$; or
\item[(b)] $\varphi_0\in H$ and $\|u - \varphi_\alpha\| \to  \|\varphi_0\| + \|u\|$ as $|\alpha| \to \infty$, for all $u\in E$ $(=H)$.
\end{itemize}
\end{itemize}
\end{assump}

It is worth noting that we do not assume that $\varphi_0\in H$ necessarily.  However, under these assumptions we have that $\varphi_\alpha - \varphi_0 \in H$ for any $\alpha\in\R$ and therefore $\varphi_\alpha + v \in E$ for all $v\in H$ and $\alpha\in\R$.

\begin{assump}
\label{assump f}
Assume that the nonlinear function $f$ acting in $E$ is such that:
\begin{itemize}
\item [(i)] $f$ is defined on all of $E$, and for all $u\in E$ there exists $f'(u) \in L(H,H)$ such that for all $v \in H$,
\[
\lim_{h\to 0}\norm{\frac{f(u + hv) - f(u)}{h} - f'(u)v} = 0;
\]
\item[(ii)] $\sup_{u\in E} \|f'(u)\|_{L(H,H)}<\infty$ (so that $H\ni v \mapsto f(\varphi_\alpha  + v)$ is globally Lipschitz $\forall \alpha\in\R$);
\item[(iii)] the map $H\ni v \mapsto f'(v + \varphi_\alpha)$ is globally Lipschitz $\forall \alpha\in\R$.
\end{itemize}
\end{assump}

\begin{assump}
\label{assump A}
Assume that the operator $A$ is such that:
\begin{itemize}
\item [(i)] The restriction of $A$ to $H$ (also denoted by $A$) is the generator of a $\mathcal{C}_0$-semigroup on $H$.  Therefore (under Assumption \ref{assump f} (i)) $\mathcal{L}_\alpha = A + f'(\varphi_\alpha):H \to H$ is also the generator of $\mathcal{C}_0$-semigroup on $H$ for all $\alpha\in\R$.
\item[(ii)] The spectrum $\sigma(\mathcal{L}_\alpha)$ of $\mathcal{L}_\alpha$ is such that 
\[
\sigma(\mathcal{L}_\alpha) \subset \{\lambda\in \mathbb{C}: \mathfrak{Re}(\lambda) + a|\mathfrak{Im}(\lambda)| \leq -b\} \cup\{0\},
\]
for some positive constants $a$ and $b$, independent of $\alpha$.  Note that by differentiating \eqref{stationary equation} with respect to $\alpha$, $0$ is always a simple eigenvalue of $\mathcal{L}_\alpha$ corresponding to eigenvector $\varphi_\alpha'$.
\end{itemize}
\end{assump}
In what follows we will make precise at the start of each section which of these assumptions are needed.  In particular, we only use  Assumption \ref{assump A} (ii) in Section \ref{sec:local stability}.

%%%%%%%%%%%%%
% Examples
%%%%%%%%%%%%%
\section{Examples}
\label{examples}
We will have two specific examples in mind that fit into this general setting: traveling fronts and pulses. These are outlined in greater detail further below. However our framework should be applicable to many other spatially-extended patterns, including Turing-type instabilities of reaction-diffusion systems, mechanical buckling or wrinkling, patterns in bacterial chemotaxis and a huge range of phenomena in neuroscience (as typically modeled using neural field equations). See \cite{murray2001mathematical} for a survey of all of the above, and \cite{bressloff-review,coombes2005waves,deco2008dynamic, ermentrout1998neural,hutt2003pattern} for a survey of applications in neuroscience.
\subsection{Traveling fronts}
\label{sec:traveling fronts}
One important example of a traveling front, that has motivated this work (and should be kept in mind throughout), is the classical neural field equation in one dimension.   This equation has the following form:
\begin{equation}
\label{NF}
\partial_tu_t(x) = -u_t(x) + \int_{\R}w(x - y)F(u_t(y))dy, \quad t\geq 0,\ x\in\R,
\end{equation}
where $w\in \mathcal{C}(\R)\cap L^1(\R)$ is the connectivity function, and $F:\R \to \R$ is a smooth and bounded sigmoid function (known as the nonlinear gain function).
It is known (see \cite{ermentrout-mcleod-93} for example) that under some conditions on the functions $w$ and $F$ (in particular that there exist precisely three solutions to the equation $x=F(x)$ at $0, a$ and $1$ with $0<a<1$), then there exists a unique (up to translations) function $\hat{u}\in\mathcal{C}^\infty(\R)$ and speed $c\in\R$ such that $u_t(x) = \hat{u}(x - ct)$ is a solution to \eqref{NF}, where  $\hat{u}$ is such that
\[
 \lim_{x\to-\infty} \hat{u}(x) =0, \qquad \lim_{x\to\infty} \hat{u}(x) = 1,
\]
so that $\hat{u}$ is indeed a wave front.   Note that in this case $\hat{u}$ itself is not in $L^2(\R)$, but it can be shown that all derivatives of $\hat{u}$ are bounded and in $L^2(\R)$.

Substituting $\hat{u}(x - ct)$ into \eqref{NF}, we see that $\hat{u}$ is such that $0= A\hat{u} + f(\hat{u})$,
where $Au := cu' - u$ and $f(u) = w*F(u)$, and $*$ denotes convolution as usual.
Moreover, due to translation invariance, we have that $\hat{u}_\alpha := \hat{u}(\cdot + \alpha)$ is also such that
\begin{align}
\label{deterministic solution}
0&= A\hat{u}_\alpha + f(\hat{u}_\alpha), \qquad\alpha\in\R.
\end{align}
We are thus in a specific situation of the general setup described in the previous section, with $H=L^2(\R)$ and $\varphi_\alpha :=\hat{u}_\alpha$.  Indeed, it is straightforward to check that Assumptions \ref{assump varphi}, \ref{assump f} and  \ref{assump A} (i) are satisfied (in particular  Assumption \ref{assump varphi} (iv) (a)) since all derivatives of $\hat{u}$ are bounded and in $L^2(\R)$.  Assumption \ref{assump A} (ii) is more difficult to check and is the subject of recent and ongoing research (we are aware for example of a forthcoming article by E. Lang and W. Stannat in this direction).  It is at least satisfied in the case where the function $F$ is replaced by the Heaviside function (see \cite{Coombes-Owen, Pinto-Jackson-Russell-Eugene, Sandstede, Zhang}). It should however be noted that one should be careful when comparing results for Heaviside functions with results for smooth sigmoid functions.  Other recent works that have studied the stability of traveling waves for smooth nonlinear gain functions $F$ include \cite{chossat-faye-rankin}.

\subsection{Traveling pulses}
One can modify the classical neural field equation \eqref{NF} to produce traveling pulse solutions in the following way.  Indeed consider the system
\begin{equation}
\label{NFwA}
\begin{cases}
\partial_tu_t = -u_t + \int_{\R}w(\cdot - y)F(u_t(y))dy -v_t, \quad t\geq 0\\
\partial_tv_t =  \theta u_t - \beta v_t,
\end{cases}
\end{equation}
where as above $F:\R \to \R$ is a smooth and bounded sigmoid function, $w\in \mathcal{C}(\R)\cap L^1(\R)$ and $\theta>0, \beta\geq0$ are some constants with $\theta<<\beta$ . This is called the neural field equation with adaptation (see for example \cite[Section 3.3]{bressloff-review} for a review).   This time we look for a solution to \eqref{NFwA} of the form $(u_t, v_t) = (\hat{u}(\cdot-ct), \hat{v}(\cdot-ct))$ for some $c\in\R$, such that $\hat{u}(x)$ and $\hat{v}(x)$ decay to zero as $x\to\pm\infty$.  Substituting this into \eqref{NFwA}, we are thus looking for a solution to the equation
\begin{equation}
\label{NFwAvector}
cU'(x) = \left(\begin{array}{cc} -1 &- 1\\ \theta  & -\beta\end{array}\right)U(x)+ f(U)(x), \quad x\in \R,
\end{equation}
where $U(x) = (\hat{u}(x), \hat{v}(x))$, and $f(U)(x) := (w*F(\hat{u})(x), 0)^T$, for all $x\in\R$.

It can be shown (see \cite[Section 3.1]{pinto-ermentrout} or \cite{faye-scheel})
%\cite[Section 2]{Pinto-Jackson-Russell-Eugene} for the illustrative case of the Heaviside function
that there exists (again under some conditions on the parameters) a smooth function $U := (\hat{u}, \hat{v})\in [L^2(\R)]^2$ and speed $c\in\R$ such that $U$ is a solution to \eqref{NFwAvector}.  Moreover $\hat{u}$ and $\hat{v}$ are both smooth functions whose derivatives are all bounded and in $L^2(\R)$.  Thus, again by translation invariance we have that $U_\alpha:= U(\cdot + \alpha)\in [L^2(\R)]^2$ is a solution to
\[
AU_\alpha + f(U_\alpha) = 0
\]
for all $\alpha\in\R$, where 
\[
AU: = cU' - \left(\begin{array}{cc} -1 &- 1\\ \theta  & -\beta\end{array}\right)U, \quad\forall U\in [L^2(\R)]^2.
\]
Once again we are thus in a specific situation of the general setup described in Section \ref{General Setting}, this time with $H=[L^2(\R)]^2$ and $\varphi_\alpha :=U_\alpha$.    Indeed, it is again straightforward to check that Assumptions \ref{assump varphi}, \ref{assump f} and  \ref{assump A} (i) are satisfied (this time $\varphi_0\in H$, so that $E=H$ and we can show that Assumption \ref{assump varphi} (iv) (b) holds).   Since $\hat{u}(x)\to0$ as $x\to\pm\infty$, we say that the solution is a traveling pulse.

Assumption \ref{assump A} (ii) is again more difficult to check but it is still satisfied in the case where the function $F$ is replaced by the Heaviside function (see again \cite{Coombes-Owen, Pinto-Jackson-Russell-Eugene, Sandstede, Zhang}).

\section{Generalized stochastic traveling wave equation}
\label{sec:stoch wave}
Suppose that $(\varphi_\alpha)_{\alpha\in\R}$, $f$ and $A$ satisfy Assumptions \ref{assump varphi}, \ref{assump f} and \ref{assump A} (i) respectively. Consider the following stochastic evolution equation
\begin{equation}
\label{main}
du_t = [Au_t + f(u_t)]dt + \varepsilon B(t)dW_t^Q,
\end{equation}
where $\varepsilon>0$ and $(W_t^Q)_{t\geq0}$ is an $H$-valued $Q$-Wiener process on the filtered probability space $(\Omega, \mathcal{F}, \{\mathcal{F}_t\}_{t\geq0}, \mathbb{P})$ with $Q$ a bounded, symmetric, non-negative definite linear operator on $H$ such that $\mathrm{Tr}(Q) < \infty$.  We work with the following assumptions on the noise:

\begin{assump}
\label{assump noise}
Assume that:
\begin{itemize}
\item[(i)] $B: [0, \infty) \to L(H, H)$ is continuous, and there exists a constant $C$ with $\|B(t)\|_{L(H, H)} \leq C$ for all $t\geq0$.
\item[(ii)] $B(t)$ is a unitary operator on $H$ for all $t\geq0$ i.e. $B(t)^{*}B(t) = \mathrm{Id}$ for all $t\geq0$.
\end{itemize}
\end{assump}

We will work in the general setting, but we will keep the three examples of Section \ref{examples} in mind.

\begin{prop}
\label{K-S}
Suppose that the (deterministic) initial condition $u_0$ is such that $v^\alpha_0:= u_0 - \varphi_\alpha \in H$ for some $\alpha\in\R$. Then stochastic evolution equation \eqref{main} has a unique solution, which can be decomposed (in a non-unique way) as $u_t = \varphi_\alpha + v^\alpha_t$ where $(v^\alpha_t)_{t\geq0}$ is the unique weak (and mild) $H$-valued solution to
\[
dv^\alpha_t  = [Av^\alpha_t+ f(\varphi_\alpha + v^\alpha_t) -  f(\varphi_\alpha)]dt + \varepsilon B(t) dW_t^Q, \quad t\geq0,
\]
with initial condition $v^\alpha_0$ i.e.
\[
v^\alpha_t  = P^A_{t}v^\alpha_0+ \int_0^tP^A_{t-s}\left[f(\varphi_\alpha + v^\alpha_s) -  f(\varphi_\alpha)\right]ds + \varepsilon\int_0^t P^A_{t-s}B(s) dW_s^Q, \quad t\geq0.
\]
where $(P^A_t)_{t\geq0}$ is the semigroup generated by $A$.
\end{prop}
\begin{proof}
The proof of this result is a straightforward application of \cite[Theorem 7.4]{D-Z} using the globally Lipschitz assumption on $f$ (Assumption \ref{assump f} (ii)), the fact that $A$ generates a $\mathcal{C}_0$-semigroup on $H$ (Assumption \ref{assump A} (i)) and the assumptions on $B$ above.  This is also a generalization of \cite[Theorem 3.1]{K-S}, though the proof is the same.
\end{proof}

\begin{rem}
We remark that for traveling waves, \eqref{main} is in the the moving coordinate frame.  To illustrate what we mean by this, suppose again we are in the concrete situation of the standard neural field equation described in Section \ref{sec:traveling fronts}, so that there is a solution $\hat{u}(x-ct)$ to \eqref{NF} for some speed $c$.  The stochastic version of this equation with purely additive noise would then be $du_t = [-u_t + w*F(u_t)]dt + dW_t^Q$.
In the moving frame (i.e. under the change of variable $x\mapsto x-ct$), the equation becomes
\[
du_t = [Au_t+ w*F(u_t)]dt + B(t)dW_t^Q,
\]
where as above $Au = cu' - u$, $u\in\mathcal{D}(A)$ and now $B(t)v := v(\cdot + ct)$ for $v\in H$. It is clear that such a $B$ clearly satisfies Assumption \ref{assump noise}.
\end{rem}

\section{Tracking the wave front}
\label{sec:Tracking the wave front}
Suppose that $(\varphi_\alpha)_{\alpha\in\R}$, $f$ and $A$ satisfy Assumptions \ref{assump varphi}, \ref{assump f} and \ref{assump A} (i) respectively. Consider the solution $(u_t)_{t\geq0}$ to \eqref{main} with initial condition $u_0$ such that $u_0 - \varphi_0\in H$ according to Proposition \ref{K-S}.  

If $\varepsilon =0$ and $u_0 = \varphi_0$, we would have that $u_t = \varphi_0$ for all $t\geq0$.  However, in the case when $\varepsilon >0$, the solution $(u_t)_{t\geq0}$ started from $\varphi_0$ will resemble a stochastic wave front, and its ``position'' will move.  In order to be able to keep track of this movement, we first have to give a precise definition of the position of the stochastic front at any time $t\geq0$.

To this end, we look for another decomposition of the solution $(u_t)_{t\geq0}$ to \eqref{main} as
\begin{equation}
\label{decomp}
u_t = z_t + \varphi_{\beta_t}, \quad t\geq0,
\end{equation}
for some general $\R$-valued stochastic process $(\beta_t)_{t\geq0}$ of bounded quadratic variation.  

Ideally, for each time $t\geq0$ we would like to choose $\beta_t$ in order to minimize the function
\begin{equation}
\label{min funct}
\alpha \mapsto m(t, \alpha):= \|u_t - \varphi_\alpha\|^2 
\end{equation}
over $\alpha\in\R$, so that $\varphi_{\beta_t}$ is then the closest of the family $\{\varphi_\alpha:\alpha\in\R\}$ of stationary solutions to the stochastic front $u_t$ in the $H$-norm.  We would then be able to say that $\beta_t$ is the \textit{position of the stochastic wave front} $u_t$ \textit{at time} $t$.

If $u_0 = \varphi_0$, it is clear that  there is a unique global minimizer of $m(0, \cdot)$, which is obtained at $0$. 
However, for times $t>0$ things are more complicated.  The following observation at least guarantees the existence of a global minimizer of the function under our conditions. 

\begin{lem}
\label{beta existence}
At every $t\geq0$ there exists at least one global minimum of the function $\alpha \mapsto m(t, \alpha) = \|u_t - \varphi_\alpha\|^2$.
\end{lem}

\begin{proof}
Suppose we are in case of Assumption \ref{assump varphi} (iv) (a). We have that for any $t\geq0$ and $\alpha\in\R$
\[
\|u_t - \varphi_\alpha\| = \|v^0_t + \varphi_0 - \varphi_\alpha\|
\]
where $(v^0_t)_{t\geq0}$ is the $H$-valued process as defined in Proposition \ref{K-S}. Thus $\|u_t - \varphi_\alpha\| \geq \|\varphi_0 - \varphi_\alpha\| - \|v^0_t\| \to \infty$
as $\alpha\to\pm\infty$, so the result holds by continuity.

On the other hand, suppose we are in case of Assumption \ref{assump varphi} (iv) (b).  Suppose (for a contradiction) that $\|u_t - \varphi_\alpha\| >  \|\varphi_0\| + \|u_t\| =  \|\varphi_\alpha\| + \|u_t\|$ for some $\alpha\in\R$.  Then by the triangle inequality
$\|\varphi_\alpha\| + \|u_t\| < \|u_t - \varphi_\alpha\| \leq  \|\varphi_\alpha\| + \|u_t\|$,
which is a contradiction.  Together with the fact that $\|u_t - \varphi_\alpha\| \to  \|\varphi_0\| + \|u_t\|$ as $\alpha\to\pm\infty$ by assumption, we again have the result.
\end{proof}
 
It is important to make two remarks at this point, both of which are illustrated in the concrete case of the traveling front solution to the
neural field equation below.  Firstly, in general we do not expect there to exist a unique global minimizer of $m(t, \cdot)$ at every time $t\geq0$.  The point is that we can have certain noises $W^Q(t, x)$ or initial conditions such that the solution $(u_t)_{t\geq0}$ to the equation \eqref{main} is at some time equally close in the $H$-norm to $\varphi_{\alpha_1}$ and $\varphi_{\alpha_2}$ with $\alpha_1\neq\alpha_2$.

The second important remark is that if $u_0 = \varphi_0$ and we continuously track the position of the initial global minimum of $m(t, \cdot)$,
as we do in the next section, 
then the noise might be such that this global minimum first becomes a \textit{local} minimum, and then might even cease to be a minimum at all (it becomes a saddle point).  Therefore any process $(\beta_t)_{t\geq0}$ attempting to keep track of a global minimum of  $m(t, \cdot)$ given by \eqref{min funct} (and hence to keep track of the position of the stochastic front) must be allowed to be \textit{discontinuous}.

In view of these two remarks we cannot simply define $\beta_t$ to be the global minimum of $m(t, \cdot)$ for all $t$.  Instead, in the next section we study the behavior of any local minimum of  $m(t, \cdot)$ up until the point at which it may become a saddle point.

\subsubsection{Illustration: The neural field equation}

Consider again the neural field equation \eqref{NF} discussed in Section \ref{examples}, but with an added continuous deterministic forcing term $t\mapsto g_t\in\mathcal{C}(\R)$ i.e.
\begin{equation}
\label{NFwdetnoise}
\partial_tu_t = -u_t + \int_{\R}w(\cdot - y)F(u_t(y))dy + g_t,
\end{equation}
for $t\geq0$.  We can simulate solutions to this equation, both in the case when $g_t(x) =0$ and $g_t(x) = 0.5\cos(t)e^{-10x^2}$, starting from the same initial condition.  The results are shown in Figure~\ref{fig1}.

\begin{figure}[!htb]
\centering
\begin{subfigure}{.5\textwidth}
  \centering
  \includegraphics[width=1.0\linewidth]{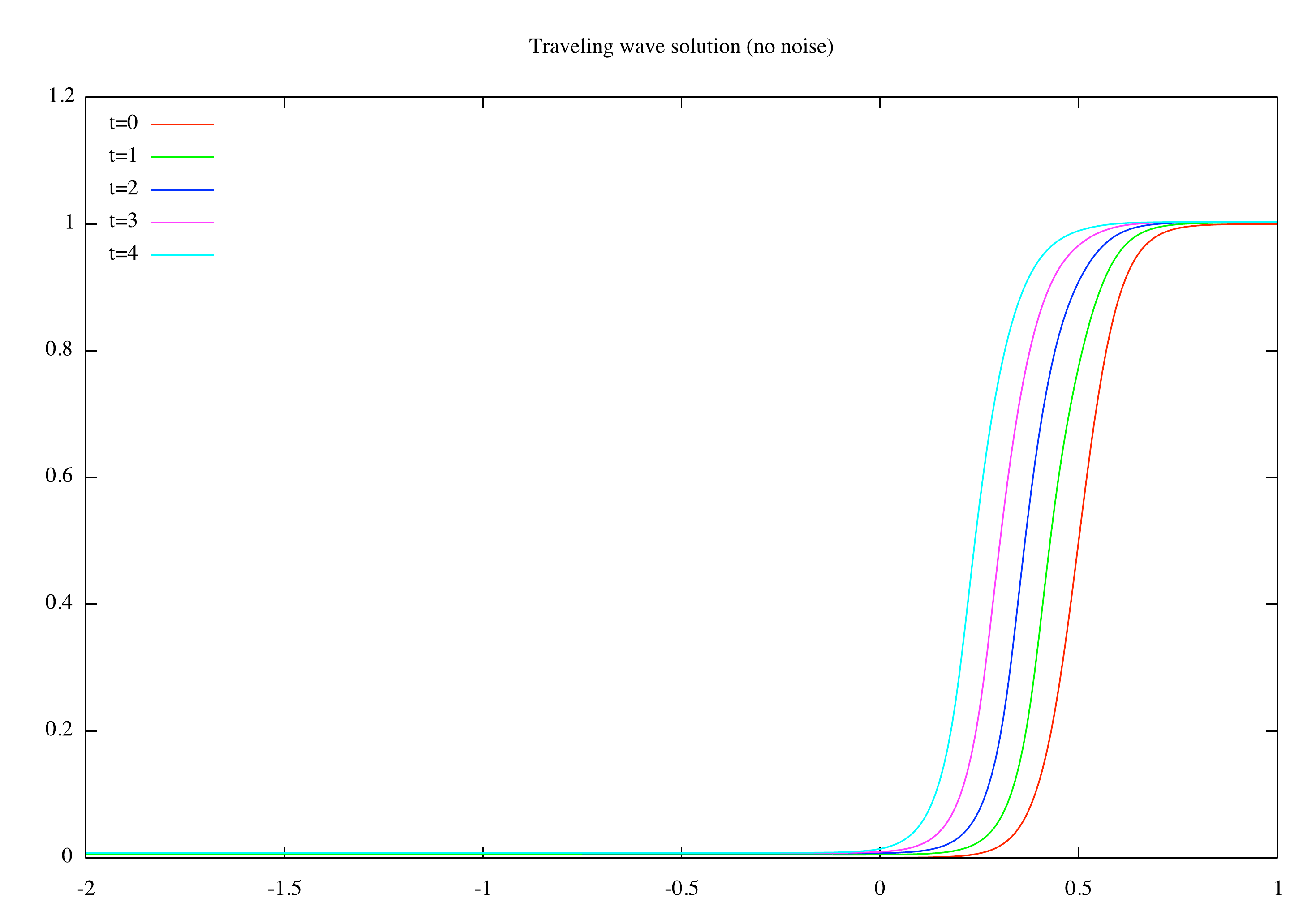}
   \end{subfigure}%
\begin{subfigure}{.5\textwidth}
  \centering
  \includegraphics[width=1.0\linewidth]{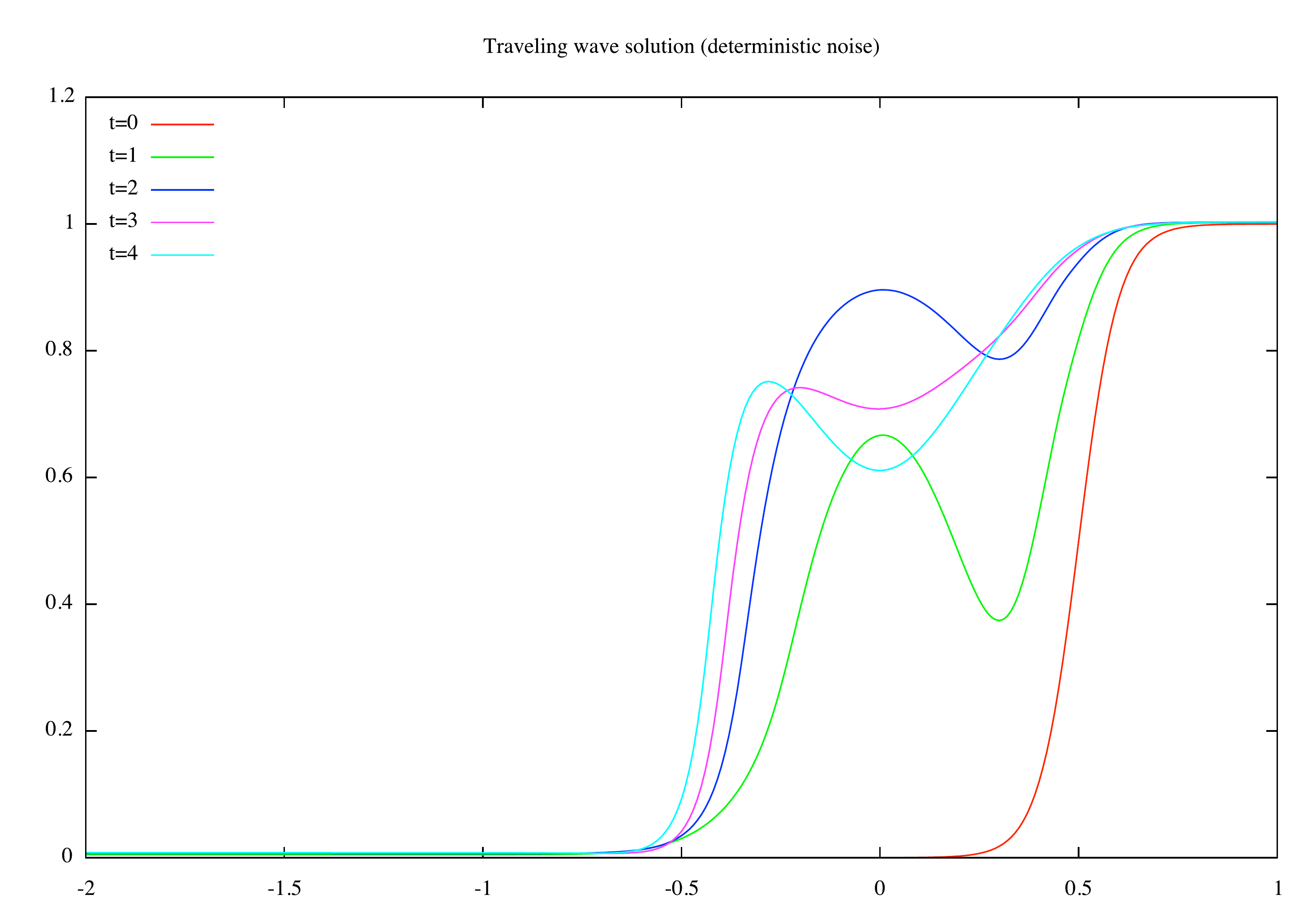}
\end{subfigure}
\caption{\label{fig1}\small{Simulations of the solution to \eqref{NFwdetnoise} with $w(x) = 10e^{-20|x|}$ and $F(x) = 0.5[1 + \mathrm{tanh}(10(x - 0.25))]$.  On the left $g_t\equiv0$, while on the right $g_t(x) = 0.5\cos(t)e^{-10x^2}$. }}
\end{figure}

We can now plot the function $\alpha \mapsto m(t,\alpha)$ given by \eqref{min funct} i.e. $\alpha\mapsto \|\varphi_\alpha - u_t\|^2$ where $(u_t)_{t\geq0}$ is a solution to \eqref{NFwdetnoise} and $g_t(x) = 0.5\cos(t)e^{-10x^2}$ (see Figure~\ref{fig2}).

\begin{figure}[!htb]
\centering
\includegraphics[width=0.6\linewidth]{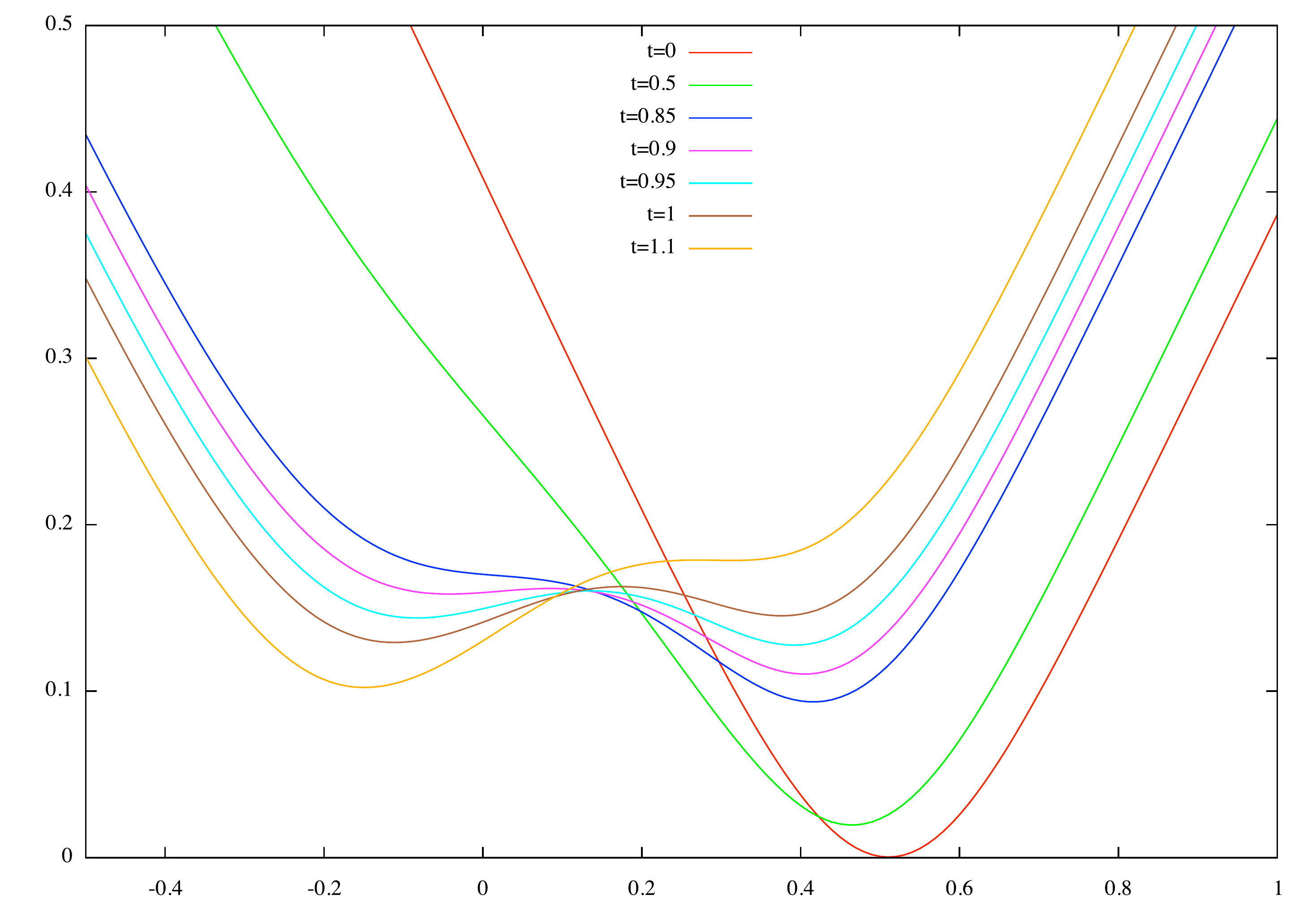}
\caption{\label{fig2}\small{Plots of the function $\alpha\mapsto\|\varphi_\alpha - u_t\|$ for different times $t$, where $(u_t)_{t\geq0}$ is a solution to \eqref{NFwdetnoise} and $g_t(x) = 0.5\cos(t)e^{-10x^2}$.}}
\end{figure}

Figure~\ref{fig2} illustrates nicely the fact that the global minimum around $\alpha=0.5$ at $t=0$ becomes a local minimum in between $t=0.95$ and $t=1.0$, and therefore that the position of the global minimum has jumped in between these times.  Moreover, at $t=1.1$ we see that the initial minimum has become a saddle point.

\subsection{The dynamics of local minima of  $\alpha\mapsto m(t, \alpha)$}
\label{sec:local dynamics} 
The aim of this section is to derive an $\R$-valued SDE that describes the behavior of any local minimum of the function $\alpha \mapsto m(t, \alpha)$ given by \eqref{min funct}, up until the point where it is no longer necessarily a local minimum.  

In order to obtain this equation, first suppose that $\beta_0$ is a local minimum of $m(0, \cdot)$.  The basic idea is then to look for a solution $\beta_t\in\R$ to
\begin{equation}
\label{zero first deriv}
\frac{d}{d\beta_t}\|u_t - \varphi_{\beta_t}\|^2 = -2\langle u_t - \varphi_{\beta_t},\varphi'_{\beta_t}\rangle =0,
\end{equation}
up until the first time $t$ when the solution is no longer necessarily a local minimum.  Such a time $t$ can be characterized by the first time that the second derivative
\[
\frac{d^2}{d\beta_t^2}\|u_t - \varphi_{\beta_t}\|^2 = -2\langle u_t ,\varphi''_{\beta_t}\rangle
\]
becomes $0$.  Although $u_t$ is not necessarily in $H$ (in particular in the traveling front case -- see Section \ref{sec:traveling fronts}), $\langle u_t ,\varphi''_{\beta_t}\rangle$
is well-defined since thanks to Proposition  \ref{K-S}, we may write $u_t = v^0_t + \varphi_0$, where $(v^0_t)_{t\geq0}$ is a well-defined $H$-valued stochastic process.  Thus (after an integration by parts)
\begin{equation}
\label{gamma}
-\langle u_t ,\varphi''_{\beta_t}\rangle = \langle\varphi'_0 ,\varphi'_{\beta_t}\rangle-\langle v^0_t,\varphi''_{\beta_t}\rangle =: \gamma(\beta_t, v_t^0),
\end{equation}
which is clearly well-defined.  Our solution to \eqref{zero first deriv} will therefore only be up until the first time that $\gamma(\beta_t, v_t^0)=0$.

The SDE describing the solution to \eqref{zero first deriv} up until this time is the following:
\begin{equation}
\label{beta}
d\beta_t = \mu(t, \beta_t,v_t^0)dt + \sigma(t, \beta_t, v_t^0)dW_t^Q, \qquad t\geq0,
\end{equation}
where $(v^0_t)_{t\geq0}$ is the $H$-valued process defined in Proposition \ref{K-S},
\begin{itemize}
\item 
\begin{equation}
\label{sigma}
\sigma(t, x, v) :=\frac{b(t, x)}{\gamma(x, v)}, 
\quad \forall x\in\R, v\in H,
\end{equation}
where $\gamma$ is defined by \eqref{gamma} and $b: \R^+\times \R \to L(H, \R)$ is given by $b(t, x)(v) = \varepsilon{\langle\varphi'_{x}, B(t)v \rangle}$ for all $v\in H$;
\item \[
\mu(t, \beta_t,v_t^0) = \sum_{k=1}^3\mu_k(t, \beta_t, v_t^0), \quad \mathrm{where}\quad \mu_k(t, x, v) := \frac{a_k(t, x, v)}{\gamma(x, v)^k},
\]
for  $x\in\R$, $v\in H$ and $k\in\{1, 2, 3\}$, where $a_k: \R^+\times\R\times H \to \R$ are functions given by
\begin{equation}
\label{a}
\begin{split}
a_1(t, x, v)&:= \langle v, A^*\varphi'_{x} \rangle + \langle f(v + \varphi_0) - f(\varphi_{0}), \varphi'_{x} \rangle \\
a_2(t, x, v) &:= \varepsilon^2 \langle B(t)QB^*(t)\varphi_{x}', \varphi_{x}'' \rangle \\
a_3(t, x, v) &:= \frac{\varepsilon^2}{2}\langle v + \varphi_0 - \varphi_{x}, \varphi_{x}''' \rangle \langle B(t)QB^*(t)\varphi_{x}', \varphi_{x}'\rangle. \\
 %&= \frac{\varepsilon^2}{2}\langle v + \varphi_0 , \varphi_{x}''' \rangle \langle B(t)QB^*(t)\varphi_{x}', \varphi_{x}'\rangle.
\end{split}
\end{equation}
\end{itemize}

Formally, the SDE \eqref{beta} can be obtained by It\^o's formula and a comparison of coefficients: if one assumes that $(\beta_t)_{t\geq0}$ satisfies an SDE driven by $(W_t^Q)_{t\geq0}$ with drift and diffusion coefficients to be determined, then by formally applying It\^o's formula, one can write down an SDE for $(\langle u_t - \varphi_{\beta_t},\varphi'_{\beta_t}\rangle)_{g\geq0}$.  Setting the result to zero (since we want a solution to \eqref{zero first deriv}) and comparing coefficients leads to \eqref{beta}.  

However, since we cannot find any (infinite-dimensional) It\^o-type lemma that directly applies to our situation, we take care in Proposition \ref{zero} below to rigorously prove the result.  In any case, we start with the following existence and uniqueness result. 
\begin{prop}
\label{existence}
Let $\tau$ be a stopping time with respect to the filtration $\{\mathcal{F}_t\}_{t\geq0}$.  For any $\mathcal{F}_{\tau}$-measurable random variable $\beta_{\tau}$ such that $\gamma(\beta_{\tau}, v^0_{\tau}) >0$ almost surely (where $\gamma$ is defined in \eqref{gamma}) and $\E(\beta_{\tau}^2)<\infty$, there exists a unique continuous solution $(\beta_t)_{t\in [\tau, \tau_\infty)}$ to the SDE \eqref{beta}, with initial condition $\beta_{\tau}$ at $\tau$, up until the stopping time $\tau_\infty = \lim_{n\to\infty} \tau_n > \tau$, where
\[
\tau_n := \inf\{t\geq \tau: \gamma(\beta_t, v^0_t) = 1/n\}.
\]
In other words
\[
\beta_{t\wedge \tau_n} = \beta_{\tau} + \sum_{k=1}^3\int_{\tau}^{t\wedge\tau_n} \mu_k(s, \beta_s, v^0_s)ds + \int_{\tau}^{t\wedge\tau_n}\sigma(s, \beta_s, v^0_s)dW^Q_s, \quad t\geq \tau,\ n\geq1.
\]

\end{prop}

\begin{proof}
The proof follows the fairly standard proof of existence and uniqueness of solutions to SDEs with locally Lipschitz coefficients up until an explosion time (see for example \cite[Theorem 1.18]{Hsu}).  We however recall the key arguments here, since we are in a slightly non-standard set-up.  
We also suppose that $\tau=0$ (the general case is the same).

\vspace{0.2cm}
\noindent\textit{Step 1: Existence.} 
Define for $t\geq0$, $x\in\R$, $v\in H$ such that $\|v\|\leq R$ and $n\geq1$,
\[
\sigma^n(t, x, v) = 
\begin{cases}
\frac{b(t, x)}{\gamma(x, v)} \ \mathrm{if}\ x\ \mathrm{is\ s.t}\ \gamma(x, v) \geq \frac{1}{n}\\
nb(t, x) \ \mathrm{otherwise}
\end{cases},
\]
and similarly
\[
\mu^n_k(t, x, v) = 
\begin{cases}
\frac{a_k(t, x, v)}{\gamma(x, v)^k} \ \mathrm{if}\ x\ \mathrm{is\ s.t}\ \gamma(x, v) \geq \frac{1}{n}\\
n^ka_k(t, x, v) \ \mathrm{otherwise}
\end{cases},
\]
for $k\in\{1, \dots, 3\}$.
Then define, for $n\geq1$, $(\beta^n_t)_{t\geq 0}$ to be the solution to the SDE
\begin{equation}
\label{beta stopped}
d\beta^n_t = \sum_{k=1}^3 \mu^n_k(t, \beta^n_t, v^0_t)dt + \sigma^n(t, \beta^n_t, v^0_t)dW^Q_t, \quad t\geq0,
\end{equation}
with initial condition $\beta^n_0 = \beta_0$.  This is a stochastic differential equation driven by a Hilbert space-valued process that fits into the standard framework of Da Prato and Zabczyk described in \cite{D-Z}.  In particular it has a unique continuous (strong) solution that does not explode up until time
\begin{equation}
\label{rho}
\rho_R := \inf\{t>0: \|v^0_t\| >R\},
\end{equation}
for all $R>0$ (note that $\rho_R$ is independent of $\beta^n$ for all $n$).  This follows from standard methods since it can be checked that $\sigma^n(t, \cdot, v)$ and $\mu^n_k(t, \cdot, v)$, $k=1,2,3$ are globally Lipschitz for $v\in H$ such that $\|v\|\leq R$ (independently of $t$), using Assumptions \ref{assump varphi} and \ref{assump f}. 
For example if $x, y\in\R$ are such that  $\gamma(x, v)\geq1/n$, $\gamma(y, v)\geq1/n$ for $v\in H$ such that $\|v\|\leq R$ then
\begin{align*}
&\|\sigma^n(t, x,v) - \sigma^n(t, y, v)\|_{L(H, \R)} = \left\|\frac{b(t, x)\gamma(y, v) - b(t, y)\gamma(x, v)}{\gamma(x, v)\gamma(y, v)}\right\|_{L(H, \R)}\\
&\qquad \leq  n^2 \|b(t, x) - b(t, y)\|_{L(H, \R)}|\gamma(y, v)| + n^2 \|b(t, y)\|_{L(H, \R)}|\gamma(y, v) - \gamma(x, v)|\\
&\qquad \leq  n^2 \varepsilon C  (1 + R)|x - y| 
\end{align*}
where $C$ depends on the Lipschitz constants of $x\mapsto \varphi'_x$ and $x\mapsto \varphi''_x$, as well as $\|\varphi'_0\|$, $\|\varphi''_0\|$ and $\sup_{t\geq0}\|B(t)\|_{L(H, H)}$.
The last inequality follows from the facts that $\|b(t, x) - b(t, y)\|_{L(H,\R)} \leq \varepsilon \|B(t)\|_{L(H, H)}\| \varphi'_x - \varphi'_y\|$, $\|b(t, x)\|_{L(H,\R)} \leq \varepsilon  \|B(t)\|_{L(H, H)}\| \varphi'_0\|$,
$|\gamma(x, v) - \gamma(y, v)| \leq  \| \varphi_0'\| \|\varphi'_{x} - \varphi'_y\| + \| v\| \|\varphi_{x}'' -  \varphi_{y}''\|$
and
$|\gamma(y, v)| \leq  \| \varphi_0'\|^2+ \| v\| \|\varphi_{0}''\|$.
The same holds if $x, y\in\R$ are such that $\gamma(x, v)\geq1/n$ and $\gamma(y, v)<1/n$, or vice-versa, and trivially holds if  $\gamma(x, v)<1/n$ and $\gamma(y, v)<1/n$.

Finally, we have that $\lim_{R\to\infty}\rho_R = \infty$ almost surely thanks to Theorem \ref{K-S}.  Thus there exists a unique continuous solution  $(\beta^n_t)_{t\geq 0}$ to \eqref{beta stopped} for all $t\geq0$.

Now, with $(\beta^n_t)_{t\geq0}$ uniquely defined by \eqref{beta stopped}, we set
\[
\tau_n := \inf\{t\geq0: \gamma(\beta^n_t, v^0_t) = n^{-1},\ \mathrm{or}\ \gamma(\beta^{n+1}_t, v^0_t) = n^{-1}\}.
\]
This makes sense because $t\mapsto \gamma(\beta^n_t, v^0_t)$ is almost surely continuous and by the conditions on $\beta_0$, $\gamma(\beta^n_0, v^0_0)>n^{-1}$ almost surely for some $n$ large enough.  We then have that $\beta^n_t = \beta^{n+1}_t$ for all $t\leq \tau_n$ since
%for $t<\tau_n\wedge\rho_R$, $\gamma(\beta^n_t, v^0_t)\geq n^{-1}$ and $\gamma(\beta^{n+1}_t, v^0_t)\geq n^{-1}$ and 
by definition $\sigma^n(t, x,  v^0_t) = \sigma^{n+1}(t, x,  v^0_t)$ for all $x$ such that $\gamma(x,  v^0_t)\geq n^{-1}$.  In other words, $(\beta^n_{t\wedge \tau_n} )_{t\geq0}$ and $(\beta^{n+1}_ {\tau_n})_{t\geq0}$ are solutions to the same equation. 
Moreover, $\tau_{n}$ is the first time that $\gamma(\beta^n_{t\wedge \tau_n}, v^0_t) = \gamma(\beta^{n+1}_{t\wedge \tau_n}, v^0_t)\leq1/n$. In particular $\tau_n \leq \tau_{n+1}$.  %Indeed, suppose that $\tau_{n+1}<\tau_n$.  Then $\gamma(\beta^{n+1}_{\tau_{n+1}\wedge \tau_n}, v^0_t) \geq1/n$ i.e. we have not yet reached $1/n$.  However, we also have that $\gamma(\beta^{n+1}_{\tau_{n+1}\wedge \tau_n}, v^0_t) = \gamma(\beta^{n+1}_{\tau_{n+1}}, v^0_t) = 1/(n+1) <  1/n$, which is a contradiction.

Let $\tau_\infty = \lim_{n\to\infty}\tau_n$.  Define $\beta_t := \beta^n_t, \quad \forall t\in[0, \tau_n)$.
Then $\tau_n$ is the first time $\gamma(\beta_t, v^0_t) \leq1/n$. Finally since
\[
\beta^n_t = \beta_0 + \sum_{k=1}^3\int_0^t \mu^n_k(s, \beta^n_s, v^0_s)ds + \int_0^t\sigma^n(s, \beta^n_s, v^0_s)dW^Q_s,
\]
together with the facts that $\beta_{t\wedge\tau_n} = \beta^n_t$, $\sigma^n(s, \beta_s^n, v^0_s) = b(s, \beta_s)/\gamma(\beta_s, v^0_s)$, and $\mu^n_k(s, \beta_s^n, v^0_s) = a_k(s, \beta_s, v^0_s)/\gamma(\beta_s, v^0_s)^k$ for all $s\leq \tau_n$ and $k\in\{1, \dots, 3\}$ we have that
\[
\beta_{t\wedge \tau_n} = \beta_0 + \sum_{k=1}^3\int_0^{t\wedge\tau_n} \mu_k(s, \beta_s, v^0_s)ds + \int_0^{t\wedge\tau_n}\sigma(s, \beta_s, v^0_s)dW^Q_s,
\]
for all $t\geq0, n\geq1$.  In other words $(\beta_t)_{t\geq0}$ is a solution to \eqref{beta} up until time $\tau_\infty$.

\vspace{0.2cm}
\noindent\textit{Step 2: Uniqueness.} 
Suppose that $(\tilde{\beta}_t)_{t\geq0}$ is another continuous solution to \eqref{beta} with initial condition $\beta_0$ up until time $\tilde{\tau}_\infty$.  Let $\rho_n$ be the first time that either $\gamma(\beta_t, v_t^0)$ or $\gamma(\tilde{\beta_t}, v_t^0)$ is equal to $1/n$ (again for $n$ large enough so that $\gamma({\beta_0}, v_0^0) > n^{-1}$).  Then $({\beta}_{t\wedge\rho_n})_{t\geq0}$ and $(\tilde{\beta}_{t\wedge\rho_n})_{t\geq0}$ are solutions to the equation \eqref{beta stopped}, so that by uniqueness of solutions to this equation, ${\beta}_{t} = \tilde{\beta}_t$ for all $t\leq \rho_n$, and $\rho_n$ is the first time that $\gamma(\beta_{t\wedge\rho_n}, v^0_t) =\gamma(\tilde{\beta}_{t\wedge\rho_n}, v^0_t) \leq 1/n$.  Hence $\tilde{\tau}_\infty = \tau_\infty = \lim_{n\to\infty}\rho_n$ and $\tilde{\beta}_t = \beta_t$ for all $t\in[0, \tau_\infty)$.
\end{proof}

\begin{prop}
\label{zero}
Let $\tau$ be a stopping time with respect to the filtration $\{\mathcal{F}_t\}_{t\geq0}$.  Let $\beta_{\tau}$ be an $\mathcal{F}_{\tau}$-measurable random variable such that $\gamma(\beta_{\tau}, v^0_{\tau}) >0$ almost surely, $\E(\beta_{\tau}^2)<\infty$ and 
\[
\langle u_\tau -  \varphi_{\beta_\tau},  \varphi_{\beta_\tau}'\rangle =0.
\] 
Then the solution  $(\beta_t)_{t\in[\tau, \tau_\infty)}$ to the SDE \eqref{beta}, as defined in Proposition \ref{existence} is such that
\[
\langle u_t- \varphi_{\beta_t}, \varphi'_{\beta_t}\rangle = 0, \quad \forall t\in[\tau, \tau_\infty).
\]
\end{prop}

\begin{proof}
The proof is a rather standard adaptation of \cite[Theorem 4.17]{D-Z}, and therefore we have not included every detail. 

Without loss of generality we may assume that $\tau=0$.  Let $\tau_n$ be as in Proposition \ref{existence}, so that $\tau_n\uparrow \tau_\infty$.  Define
\[
\xi_n := \inf\left\{ t\in[0, \tau_n]: \norm{\int_0^t B(s)dW^Q_s}\geq n\ \text{ or } \ |\beta_t | \geq n\right\}, 
\]
with $\xi_n = \tau_n$ if the set is empty. It may be seen that $(\xi_n)_{n\geq1}$ is nondecreasing, and that $\lim_{n\to\infty}\xi_n = \tau_\infty$ a.s. Define for any $t\geq0$ $\beta^{n}_t = \beta_{t\wedge \xi_n}$ and $v^{0,n}_t = v^0_{t\wedge \xi_n}$ where as above $(v^0_t)_{t\geq0}$ is the $H$-valued solution to the SDE in Proposition \ref{K-S} (with $\alpha=0$). 
Let $\Delta(v,\alpha): H\times \R \to \R := - \langle v+\varphi_0 -  \varphi_{\alpha},  \varphi_{\alpha}'\rangle$ and $\Pi = (t_i)_{i=1}^M$ be a partition of $[0,t]$ for some $t\geq0$. For some family $\lbrace \theta_k\rbrace_{k=1}^{M-1} \subset [0,1]$ to be specified below, set $w_k = \theta_k v^{0,n}_{t_{k}} + (1-\theta_k) v^{0,n}_{t_{k+1}}$ and $\zeta_k = \theta_k \beta^{n}_{t_{k}} + (1-\theta_k) \beta^{n}_{t_{k+1}}$. 
Let $X_k = (v^{0,n}_{t_{k+1}}-v^{0,n}_{t_k},\beta^{n}_{t_{k+1}}-\beta^{n}_{t_k})$. We note that the double Fr\'echet derivative of $\Delta$, evaluated at $(w_k,\zeta_k)$, and in the direction $X_k$  is
\begin{multline}
\label{second frechet deriv}
D^2\Delta(w_k,\zeta_k)\cdot X_k \cdot X_k = -2\langle v^{0,n}_{t_{k+1}}-v^{0,n}_{t_k},\varphi''_{\zeta_k}\rangle (\beta^{n}_{t_{k+1}}-\beta^{n}_{t_k})\\ 
-\langle w_k + \varphi_0 - \varphi_{\zeta_k},\varphi'''_{\zeta_k}\rangle(\beta^{n}_{t_{k+1}}-\beta^{n}_{t_k})^2.
\end{multline}
By Taylor's theorem,
\begin{align}
\label{taylor}
\Delta(v^{0,n}_t,\beta^{n}_t) -\Delta(v^{0}_0,\beta_0) &= \sum_{k=1}^{M-1}\gamma(\beta^{n}_{t_k},v^{0,n}_{t_k}) (\beta^{n}_{t_{k+1}} - \beta^{n}_{t_k})- \langle v^{0,n}_{t_{k+1}} -v^{0,n}_{t_k},\varphi'_{\beta^{n}_{t_{k}}}\rangle  \nonumber\\
&\quad + \frac{1}{2}\sum_{k=1}^{M-1} D^2\Delta(w_k,\zeta_k)\cdot X_k \cdot X_k,
\end{align}
for some $\lbrace \theta_k\rbrace_{k=1}^{M-1} \subset [0,1]$. As the partition $\Pi\to0$, we find thanks to Proposition \ref{K-S} that 
\begin{multline*}
\sum_{k=1}^{M-1} \langle v^{0,n}_{t_{k+1}} - v^{0,n}_{t_k},\varphi'_{\beta^{n}_{t_{k}}}\rangle \\
\to \int_{0}^{\xi_n\wedge t}[\langle A^*\varphi_{\beta_{s}}',v^{0}_s\rangle + \langle f(v^{0}_s+\varphi_0) - f(\varphi_0),\varphi'_{\beta_{s}}\rangle ]ds +\varepsilon \int_{0}^{\xi_n\wedge t}\langle\varphi'_{\beta_{s}},B(s)dW^Q_s\rangle. 
\end{multline*}
Similarly, thanks to \eqref{beta},
\begin{multline*}
\sum_{k=1}^{M-1}\gamma(\beta^{n}_{t_k},v^{0,n}_{t_k})(\beta^{n}_{t_{k+1}} - \beta^{n}_{t_k})\to \varepsilon \int_{0}^{t\wedge\xi_n}\langle\varphi'_{\beta_{s}},B(s)dW^Q_s\rangle + \sum_{l=1}^3\int_{0}^{t\wedge\xi_n}\frac{a_l(s, \beta_s, v^0_s)}{\gamma(\beta_s, v_s^0)^{l-1}}ds,
\end{multline*}
as $\Pi\to0$.
We then have to deal with the second order terms in the Taylor expansion \eqref{taylor}.  According to \eqref{second frechet deriv}, there remain two terms on the right-hand side of \eqref{taylor} to handle:
\begin{align*}
-\sum_{k=1}^{M-1} \langle v^{0,n}_{t_{k+1}}-v^{0,n}_{t_k},\varphi''_{\zeta_k}\rangle (\beta^{n}_{t_{k+1}}-\beta^{n}_{t_k})\ \textrm{and}\ -\frac{1}{2}\sum_{k=1}^{M-1}\langle w_k + \varphi_0 - \varphi_{\zeta_k},\varphi'''_{\zeta_k}\rangle(\beta^{n}_{t_{k+1}}-\beta^{n}_{t_k})^2.
\end{align*}
For the first of these terms, by again using Proposition \ref{K-S} and \eqref{beta} it is standard to show that
\begin{align*}
&\lim_{\Pi\to 0}\sum_{k=1}^{M-1} \langle v^{0,n}_{t_{k+1}}-v^{0,n}_{t_k},\varphi''_{\zeta_k}\rangle (\beta^{n}_{t_{k+1}}-\beta^{n}_{t_k}) \\
%&=\varepsilon^2\lim_{\Pi\to 0}\sum_{k=1}^{M-1}\int_{t_k\wedge\xi_n}^{t_{k+1}\wedge\xi_n}\langle\varphi''_{\zeta_k},B(s)dW^Q_s\rangle  \int_{t_k\wedge\xi_n}^{t_{k+1}\wedge\xi_n}\frac{\langle\varphi'_{\beta_{s}},B(s)dW^Q_s\rangle}{\gamma(\beta_s, v^0_s)}\\
&=\varepsilon^2\lim_{\Pi\to 0}\sum_{k=1}^{M-1}\int_{t_k\wedge\xi_n}^{t_{k+1}\wedge\xi_n}\langle\varphi''_{\beta^n_{t_k}},B(s)dW^Q_s\rangle  \int_{t_k\wedge\xi_n}^{t_{k+1}\wedge\xi_n}\frac{\langle\varphi'_{\beta_{s}},B(s)dW^Q_s\rangle}{\gamma(\beta_s, v^0_s)} =:\varepsilon^2\lim_{\Pi\to 0}J(\Pi).
\end{align*}
almost surely.  In order to find this limit, following the standard method to prove It\^o's lemma (see \cite[Theorem 4.17]{D-Z} or \cite[Theorem 3.3.3]{K-S}) and using the infinite dimensional It\^o isometry (see \cite[Theorem 4.12]{D-Z})we have
\[
\E\left[\left(J(\Pi)-\int_{0}^{t\wedge\xi_n}\frac{\langle\varphi''_{\beta_s},B(s)QB^*(s)\varphi'_{\beta_s}\rangle}{\gamma(\beta_s,v^{0}_s)}ds\right)^2\right] \to 0
\]
as $\Pi \to 0$. We establish through an analogous argument that
\begin{align*}
&\lim_{\Pi\to 0}\sum_{k=1}^{M-1} \langle w_k + \varphi_0 - \varphi_{\zeta_k},\varphi'''_{\zeta_k}\rangle(\beta^{n}_{t_{k+1}}-\beta^{n}_{t_k})^2 \\
&= \varepsilon^2\int_0^{t\wedge \xi_n}\frac{\langle v^0_s + \varphi_0 - \varphi_{\beta_s} , \varphi_{\beta_s}''' \rangle}{\gamma(\beta_s,v^0_s)^2} \langle B(s)QB^*(s)\varphi_{\beta_s}', \varphi_{\beta_s}'\rangle ds,
\end{align*}
almost surely. Finally, by taking the size of the partition $\Pi$ to 0 in \eqref{taylor}, and using the definitions of the functions $a_1, a_2$ and $a_3$ given in \eqref{a}, we see that for any $t\leq \xi_n$, $\Delta(v^{0}_t,\beta_t) =  \Delta(v^{0}_0,\beta_0)$ so that (since $u_t = v^0_t + \varphi_0$)
\[
\langle u_t-  \varphi_{\beta_t},  \varphi_{\beta_t}'\rangle = \langle v^{0}_t+\varphi_0 -  \varphi_{\beta_t},  \varphi_{\beta_t}'\rangle = \langle v^{0}_0+\varphi_0 -  \varphi_{\beta_0},  \varphi_{\beta_0}'\rangle = 0,  \quad t\in[0, \xi_n],
\]
by assumption.  Since this holds for any $n$ and $\xi_n\uparrow \tau_\infty$ we have the result.
\end{proof}

\begin{cor}
\label{non-explosion}
Suppose that we are in the situation of Proposition \ref{zero}, and $(\beta_t)_{t\in[\tau, \tau_\infty)}$ is again the solution to the SDE \eqref{beta} up until time $\tau_\infty$, as defined in Proposition \ref{existence}. Then $\limsup_{t\to\tau_\infty}|\beta_t| <\infty$.  
Moreover, suppose that the probability that
\begin{equation}
\label{eqn ut girsanov}
\langle u_{\tau_\infty} - \varphi_\alpha,\varphi_\alpha'\rangle = 0, \quad \forall \alpha\in I,
\end{equation}
for any interval $I\subset \R$ with nonempty interior is zero. Then $\lim_{t\to\tau_\infty}\beta_t$ exists almost surely.  
\end{cor}
\begin{rem}
The assumption \eqref{eqn ut girsanov} in the above Corollary ensures that, with probability 1, the function $\alpha\mapsto m(\tau_\infty, \alpha) = \|u_{\tau_\infty} - \varphi_{\alpha}\|^2$ is not `flat' over a nonempty interval.  If this function did become flat at $\tau_\infty$, it is natural that $\lim_{t\to\tau_\infty}\beta_t$ would be undefined.

We expect that in most applications, it is impossible that there exists an interval $I\subset\R$ with nonempty interior such that 
\[
\langle u - \varphi_\alpha,\varphi_\alpha'\rangle = 0, \quad \forall \alpha\in I,
\]
whenever $u\in E$. For example, by differentiating with respect to $\alpha$ an arbitrary number of times and assuming smoothness, this would be impossible if \linebreak$u\in \bar{\mathrm{sp}}\{\varphi_\alpha, \varphi'_{\alpha}, \varphi''_{\alpha}, \dots, \alpha\in I\}$. 
\end{rem}

\begin{proof}[Proof of Corollary \ref{non-explosion}] 
 Without loss of generality, suppose $\tau=0$.  We firstly prove that almost surely $\limsup_{t\to\tau_{\infty}} |\beta_t|<\infty$. Assume for a contradiction that for a set of paths of nonzero measure, $\limsup_{t\to\tau_{\infty}} |\beta_t|=\infty$. Then, thanks to Assumption  \ref{assump varphi} (iv), and the continuity of $t\mapsto u_t$ for all $t\geq0$, for any $\varepsilon\geq0$ we can find a sequence of times  $(\xi^k)_{k\geq1}$ such that 
 \begin{itemize}
 \item$\xi^k\uparrow\tau_\infty$ as $k\to\infty$,
\item $\|v^0_{\tau_\infty} - v^{0}_{\xi^k}\| \leq \frac{\varepsilon}{k}$ for all $k\geq1$,
\item $\langle\varphi'_{\beta_{\xi^{k}}},\varphi'_{\beta_{\xi^j}}\rangle | \leq \frac{\varepsilon}{k}$ for all $j \in\{1, \dots, k-1\}$, $k\geq1$,
\item $ |\langle \varphi'_{\beta_{\xi^{k}}},\varphi_{\beta_{\xi^{k}}}- \varphi_{0}\rangle| \geq \kappa$ for all $k\geq1$ and some $\kappa>0$.
\end{itemize}
Now define $(e_k)_{k=1}^M$ to be part of an orthogonal basis for $H$ for some $M\geq1$ to be chosen later, defined through the Gram-Schmidt procedure and based on the functions $(\varphi'_{\beta_{\xi^k}})_{k=1}^M$. 
That is, $e_1 =\varphi'_{\beta_{\xi^1}}$ and
\[
e_k =\varphi'_{\beta_{\xi^k}} - \sum_{j=1}^{k-1}\frac{\langle {e}_j, \varphi'_{\beta_{\xi^k}}\rangle}{\norm{{e}_j}^2}{e}_j, \quad \mathrm{so\ that}\quad \norm{v^0_{\tau_\infty}}^2 \geq \sum_{k=1}^M\frac{\langle v^0_{\tau_\infty}, e_k\rangle^2}{\|e_k\|^2}.
\]
Now
\begin{align*}
\langle v^0_{\tau_\infty}, e_k\rangle %&= \langle v^0_{\tau_\infty}, e_k - \varphi'_{\beta_{\xi^k}}\rangle + \langle v^0_{\tau_\infty},\varphi'_{\beta_{\xi^k}}\rangle \\ 
%&= \langle v^0_{\tau_\infty}, e_k - \varphi'_{\beta_{\xi^k}}\rangle + \langle v^0_{\tau_\infty} - \varphi_{\beta_{\xi^k}},\varphi'_{\beta_{\xi^k}}\rangle + \langle \varphi_{\beta_{\xi^k}},\varphi'_{\beta_{\xi^k}}\rangle\\
&= \langle v^0_{\tau_\infty}, e_k - \varphi'_{\beta_{\xi^k}}\rangle + \langle v^0_{\tau_\infty} -v^0_{\xi^k},\varphi'_{\beta_{\xi^k}}\rangle +\langle v^0_{\xi^k} - \varphi_{\beta_{\xi^k}} +  \varphi_{\beta_{\xi^k}},\varphi'_{\beta_{\xi^k}}\rangle\\
&= \langle v^0_{\tau_\infty}, e_k - \varphi'_{\beta_{\xi^k}}\rangle + \langle v^0_{\tau_\infty} -v^0_{\xi^k},\varphi'_{\beta_{\xi^k}}\rangle +\langle u_{\xi^k} - \varphi_0 - \varphi_{\beta_{\xi^k}} +  \varphi_{\beta_{\xi^k}},\varphi'_{\beta_{\xi^k}}\rangle \\
%&= \langle v^0_{\tau_\infty}, e_k - \varphi'_{\beta_{\xi^k}}\rangle + \langle v^0_{\tau_\infty} -v^0_{\xi^k},\varphi'_{\beta_{\xi^k}}\rangle +\langle u_{\xi^k}  - \varphi_{\beta_{\xi^k}},\varphi'_{\beta_{\xi^k}}\rangle -\langle \varphi_0,\varphi'_{\beta_{\xi^k}}\rangle + \langle \varphi_0,\varphi'_0\rangle\\
&= \langle v^0_{\tau_\infty}, e_k - \varphi'_{\beta_{\xi^k}}\rangle + \langle v^0_{\tau_\infty} -v^0_{\xi^k},\varphi'_{\beta_{\xi^k}}\rangle  + \langle \varphi_{\beta_{\xi^k}} -\varphi_0, \varphi'_{\beta_{\xi^k}} \rangle
\end{align*}
since $\langle u_{\xi^k}  - \varphi_{\beta_{\xi^k}},\varphi'_{\beta_{\xi^k}}\rangle =0$.  Therefore
\begin{align}
\label{vtuainfty}
\langle v^0_{\tau_\infty}, e_k\rangle \geq |\langle \varphi'_{\beta_{\xi^k}} , \varphi_{\beta_{\xi^k}} -\varphi_0\rangle|- \|v^0_{\tau_\infty}\|\|e_k - \varphi'_{\beta_{\xi^k}}\| - \|v^0_{\tau_\infty} -v^0_{\xi^k}\|\|\varphi'_0\|.
\end{align}
Now, by definition of $e_k$, for $k\in\{1, \dots, M\}$,
\begin{multline*}
\norm{e_k - \varphi'_{\beta_{\xi^k}}}^2 = \sum_{j=1}^{k-1}\langle\varphi'_{\beta_{\xi^k}}, e_j \rangle^2
\leq  2\sum_{j=1}^{k-1}\langle\varphi'_{\beta_{\xi^k}},\varphi'_{\beta_{\xi^j}}\rangle^2 + 2\|\varphi'_0\|^2\sum_{j=1}^{k-1} \|e_j - \varphi'_{\beta_{\xi^j}}\|^2\\
\leq  2\varepsilon^2\sum_{j=1}^{\infty}\frac{1}{k^2} + 2\|\varphi'_0\|^2\sum_{j=1}^{k-1} \|e_j - \varphi'_{\beta_{\xi^j}}\|^2 = \varepsilon^2 C + 2\|\varphi'_0\|^2\sum_{j=1}^{k-1} \|e_j - \varphi'_{\beta_{\xi^j}}\|^2,
\end{multline*}
where we have used our choice of $\xi^k$, and $C = 2\sum_{j=1}^{\infty}\frac{1}{k^2} $.  By the discrete Gronwall inequality,
this yields 
\begin{align*}
\norm{e_k - \varphi'_{\beta_{\xi^k}}}^2 &\leq \varepsilon^2 C + \sum_{j=1}^{k-1} 2\varepsilon^2 C \|\varphi'_0\|^2\exp\left(\sum_{i=j+1}^{k-1} 2\|\varphi'_0\|^2\right)\leq  \varepsilon^2 CMe^{CM},
\end{align*}
for a new constant $C$ depending on $\|\varphi'_0\|$ but independent of $\varepsilon$ and $M$.  Returning now to 
\eqref{vtuainfty}, we see that agin thanks to our choice of $\xi^k$,
\[
\langle v^0_{\tau_\infty}, e_k\rangle \geq \kappa- \|v^0_{\tau_\infty}\|\varepsilon CMe^{CM} - \frac{\varepsilon}{k^2} \|\varphi'_0\| \geq \kappa/2
\]
for $\varepsilon$ small enough so that $ \varepsilon(\|v^0_{\tau_\infty}\| CMe^{CM} + \varepsilon \|\varphi'_0\|) \leq \kappa/2$.
Thus
\[
\norm{v^0_{\tau_\infty}}^2 \geq \sum_{k=1}^M\frac{\langle v^0_{\tau_\infty}, e_k\rangle^2}{\|e_k\|^2} \geq \frac{\kappa^2}{4}\sum_{k=1}^M\frac{1}{\|e_k\|^2}.
\]
This is clearly a contradiction if we take $M$ large enough since $\| e_k \| \to \| \varphi_0 \|$ as $k\to\infty$ (which implies that $\sum_{k=1}^M\|e_k\|^{-2}\to\infty$ as $M\to\infty$).

We now prove that $\lim_{t\to \tau_{\infty}}\beta_t$ exists almost surely under assumption \eqref{eqn ut girsanov}. Fix $\underline{\alpha}< \bar{\alpha}$. 
Define
\[
\xi^0 := \inf\{t\in[0, \tau_\infty): \beta_t\geq\bar{\alpha}\},
\]
and
\[
 \xi^{2k+1} := \inf\{t\in(\xi^{2k}, \tau_\infty): \beta_t\leq\underline{\alpha}\}, \quad  \xi^{2k+2} := \inf\{t\in(\xi^{2k+1}, \tau_\infty): \beta_t\geq\bar{\alpha}\}, 
\]
for all $k\geq0$, with $\inf\{\emptyset\} = \tau_\infty$ by convention.  Suppose for a contradiction that $ \xi^{n}<\tau_\infty$ for all $n$.

Let
\[
s_n = \sup_{\theta \in [\underline{\alpha},\bar{\alpha}]}\left( |\langle u_{\xi^n}-\varphi_{\theta},\varphi'_{\theta}\rangle|\right),
\] 
and let $\theta^*_n\in [\underline{\alpha},\bar{\alpha}]$ be such that this supremum is attained (this exists by continuity). 

Suppose for contradiction that  $s_n \not\to 0$ as $n\to \infty$ i.e. that there exists a subsequence $(s_{n_r})_{r=1}^\infty$ such that for some $\delta > 0$, $s_{n_r} \geq \delta$ for all $r\geq1$. We know that for all $r\geq1$ there exists some $t\in [\xi^{n_r},\xi^{n_r+1}]$,  such that $\langle u_{t}-\varphi_{\theta^*_{n_r}},\varphi'_{\theta^*_{n_r}}\rangle = 0$.  This is because at time 
$\xi^{n_r}$, $\bar{\alpha}$ (or $\underline{\alpha}$) is a local minimum of the function $\alpha\mapsto\|u_{\xi^{n_r}} - \varphi_\alpha\|^2$ while at time $\xi^{n_r+1}$ $\underline{\alpha}$ (or $\bar{\alpha}$) is.  Therefore, by continuity, since $\theta^*_{n_r}\in[\underline{\alpha}, \bar{\alpha}]$, there must exist $t\in [\xi^{n_r},\xi^{n_r+1}]$ such that $\theta^*_{n_r}$ is a local minimum of this function. 
In this case 
\begin{align*}
\sup_{r,s\in [\xi^{n_r},\xi^{n_r+1}]} \norm{u_{r} - u_{s}}^2 &\geq  \norm{u_{t} - u_{\xi^{n_r}}}^2\geq \frac{1}{\norm{\varphi'_0}^2}\langle u_t - u_{\xi^{n_r}},\varphi_{\theta^*_{n_r}}'\rangle^2 \\
&=\frac{1}{\norm{\varphi'_0}^2}\langle u_t-\varphi_{\theta^*_{n_r}} + \varphi_{\theta^*_{n_r}} - u_{\xi^{n_r}},\varphi_{\theta^*_{n_r}}'\rangle^2 \\
&= \frac{1}{\norm{\varphi'_0}^2}\langle  u_{\xi^{n_r}}-\varphi_{\theta^*_{n_r}} ,\varphi_{\theta^*_{n_r}}'\rangle^2 = \frac{s_{n_r}^2}{\|\varphi'_0\|^2} \geq \frac{\delta^2}{\norm{\varphi'_0}^2}.
\end{align*}
This is a contradiction since we are assuming that $ \xi^{n}<\tau_\infty$ for all $n$.  Indeed it implies infinite oscillations (of a nontrivial magnitude) of $\norm{u_t}$ over a compact time interval.  We can therefore conclude that $s_n \to 0$ as $n\to \infty$. 

By the continuity of $u_t$, we thus see that for all $\alpha \in [\underline{\alpha},\bar{\alpha}]$,
\[
\langle u_{\xi^\infty}-\varphi_{\alpha},\varphi'_{\alpha}\rangle =0, 
\]
where $\xi^\infty = \lim_{n\to\infty}\xi^n$.  In fact it is easy to see that $\xi^\infty = \tau_\infty$ (the process $(\beta_t)_{t\in[0, \tau_\infty)}$ cannot oscillate infinitely often before $\tau_\infty$ by continuity).  Therefore we have that
\[
\langle u_{\tau_\infty}-\varphi_{\alpha},\varphi'_{\alpha}\rangle =0, 
\]
for all $\alpha\in[\underline{\alpha}, \bar{\alpha}]$.  This event occurs with probability zero by assumption \eqref{eqn ut girsanov}, which implies that the event $\{\xi^n<\tau_\infty,\ \forall n\}$ also occurs with probability zero, proving the result.
\end{proof}

\subsection{Comparison with previous work}
We include this section to make the explicit comparison between the dynamics of the local minimum of \eqref{min funct} we describe in Section \ref{sec:local dynamics} and the work of \cite{bressloff-webber} and \cite{kruger-stannat}.  Both of these articles work in the specific case of the stochastic version of the classical neural field equation (see Section \ref{sec:traveling fronts}). 

In \cite{bressloff-webber} a formal expansion in $\varepsilon$ is used to try and deduce the dynamics of the position of the stochastic wave front, and the conclusion is that it is essentially Brownian to first order in $\varepsilon$ (see \cite[Equation (2.25) and (2.26)]{bressloff-webber}).  With our approach and definition of the position of the stochastic wave front, we allow for the fact that the position may jump.  Moreover, before the time of the first jump we can also formally expand $\beta_t$ with respect to $\varepsilon$, where $t<\tau_\infty$ and $(\beta_t)_{t\in[0, \tau_\infty)}$ is the solution to \eqref{beta} according to Proposition \ref{existence} (assume that $u_0 = \varphi_0$ so that $\beta_0 = 0$).  Indeed, by \eqref{beta}
\begin{align*}
\beta_t = \int_0^t\frac{ \langle v^0_s, A^*\varphi'_{\beta_s} \rangle + \langle f(v^0_s + \varphi_0) - f(\varphi_{0}), \varphi'_{\beta_s} \rangle}{\gamma(\beta_s, v^0_s)}ds + \varepsilon \int_0^t\frac{\langle \varphi'_{\beta_s}, B(s)dW_s^Q\rangle}{\gamma(\beta_s, v^0_s)} + \mathcal{O}(\varepsilon^2).
\end{align*}
Now, by Proposition \ref{K-S} we see that formally $v^0_t = \mathcal{O}(\varepsilon)$ and by the definition of $\gamma$ in \eqref{gamma}, this implies that $\gamma(\beta_t, v^0_t)^{-1} =  1/\|\varphi_0'\|^2 + \mathcal{O}(\varepsilon)$ for $t<\tau_\infty$.  Thus
\begin{align*}
\beta_t = \frac{1}{\|\varphi_0'\|^2} \int_0^t\langle v^0_s, \mathcal{L}_0^*\varphi'_{0} \rangle ds + \frac{\varepsilon}{\|\varphi_0'\|^2} \int_0^t\langle \varphi'_{\beta_s}, B(s)dW_s^Q\rangle+ \mathcal{O}(\varepsilon^2)
\end{align*}
where $\mathcal{L}_{0} := A + f'(\varphi_{0})$.  In our setup this formula would replace \cite[Equation (2.26)]{bressloff-webber}.  The reason for the difference is the choice of Hilbert space $H$.  Indeed, as pointed out to us by E. Lang, if we instead defined $\beta_t$ to minimize the function $\alpha\mapsto \|u_t - \varphi_{\alpha}\|^2_{L^2(\rho_\alpha)}$ with the weight $\rho_\alpha := \Psi_\alpha/\varphi'_\alpha$, where $\Psi_\alpha$ is a vector in the null space of  $\mathcal{L}_\alpha^*$, then we would (to a first order approximation in $\varepsilon$) arrive at \cite[Equation (2.26)]{bressloff-webber}.  For further details, as well as other reasons why this weight seems to be a natural one, we refer to the forthcoming PhD thesis of E. Lang.

In \cite{kruger-stannat}, the idea of minimizing $\alpha \mapsto \|u_t - \varphi_\alpha\|^2$ is used as we do to keep track of the position of the stochastic front.  However, rather than describing the dynamics of the minima of $\alpha \mapsto \|u_t - \varphi_\alpha\|^2$ explicitly, a gradient-descent adaptation procedure is proposed, whereby $(\beta_t)_{t\geq0}$ in \eqref{decomp} is defined via an ODE to converge dynamically towards the nearest local minimum with a certain speed.  As such, our solution to the SDE \eqref{beta} should be recovered by this adaptation procedure with infinite speed.

%%%%%%%%%%%%%%%%%%%%%%
% Local stability
%%%%%%%%%%%%%%%%%%%%%%

\section{Local stability}
\label{sec:local stability}
Once again suppose that $(\varphi_\alpha)_{\alpha\in\R}$, $f$ and $A$ satisfy Assumptions \ref{assump varphi}, \ref{assump f} and \ref{assump A} (i) respectively, but now suppose also that Assumption \ref{assump A} (ii) is satisfied.
Again let $(u_t)_{t\geq0} = (v^0_t + \varphi_0)_{t\geq0}$ be the solution to \eqref{main} with (deterministic) initial condition $u_0$ such that $u_0 - \varphi_0\in H$ according to Proposition \ref{K-S}.  

Suppose that $(\beta_t)_{t\in[0, \tau_\infty)}$ is the solution to the SDE \eqref{beta} with initial condition $\beta_0$ such that $\langle u_0 - \varphi_{\beta_0}, \varphi'_{\beta_0}\rangle =0$ and $\gamma(\beta_0, v^0_0) >0$ (i.e. $\beta_0$ is a local minimum of $m(0, \cdot)$ given by \eqref{min funct}), where $\tau_\infty = \inf\{t>0: \gamma(\beta_t, v^0_t) =0\}$ (see Proposition \ref{existence}).

Recall from \eqref{decomp} that $z_t$ is defined by
\begin{equation}
\label{z}
z_t = u_t - \varphi_{\beta_t}, \quad t\in[0,\tau_\infty).
\end{equation}
Then it is easy to see that $(z_t)_{t\in[0, \tau_\infty)}$ satisfies the stochastic evolution equation 
\begin{equation}
\label{zeq}
dz_t = [\mathcal{L}_\alpha z_t + G(z_t, \beta_t, \alpha)]dt + \varepsilon B(t)dW^Q_t - d\varphi_{\beta_t}
\end{equation}
for any $\alpha\in\R$, where
\begin{equation}
\label{G}
G(z, \beta, \alpha) := f(z +\varphi_{\beta}) - f(\varphi_{\beta}) - f'(\varphi_{\alpha})z,\quad \forall z\in H, \alpha, \beta\in\R,
\end{equation}
and $\mathcal{L}_\alpha$ is the operator defined by
\begin{equation}
\label{L alpha}
\mathcal{L}_{\alpha}z := Az + f'(\varphi_{\alpha})z, \quad \forall z\in\mathcal{D}(A) = \mathcal{D}(\mathcal{L}_\alpha).
\end{equation}
Let $(U_\alpha(t))_{t\geq0}$ be the $\mathcal{C}_0$-semigroup generated by $\mathcal{L}_\alpha$.  Note that $\mathcal{L}_\alpha$ does indeed generate a $\mathcal{C}_0$-semigroup, since by Assumption \ref{assump A} (i) $A$ generates a $\mathcal{C}_0$-semigroup and $f'(\varphi_{\alpha}): H \to H$ is bounded (see \cite[Theorem 1.3, Chapter III]{engel-nagel:01}).  Moreover thanks to Assumption \ref{assump A} (ii) on the operator $A$, we have the following result found in \cite[Lemma 1.2, Chapter 5]{volpert-volpert}.

\begin{lem}
\label{U decomp}
For any $t\geq0$ and $\alpha\in\R$, $U_\alpha(t)$ can be decomposed as
\begin{equation*}
U_\alpha(t) = P_\alpha + V_\alpha(t),
\end{equation*}
where $P_\alpha$ is the projection operator onto the subspace of $H$ spanned by $\varphi'_\alpha$ and $(V_{\alpha}(t))_{t\geq0}$ is a semigroup on $H$ such that for some $b>0$ it holds that 
\begin{equation*}
\norm{V_{\alpha}(t)} \leq \exp(-bt),
\end{equation*}
for any $t\geq 0$ and $\alpha\in\R$.
\end{lem}

The first result of the section is the following, which helps us understand the dynamics of the process $(\|z_t\|^2)_{t\geq0}$.

\begin{thm}
\label{thm: ineq}
For any $t \in [0, \tau_\infty)$, it holds that
\begin{align}
\label{ineq}
d\norm{z_t}^2 &\leq - b\norm{z_t}^2 dt+2\varepsilon\langle z_t,B(t)dW^Q_t\rangle  + 2\left\langle z_t, G(z_t,\beta_t,\beta_t) \right\rangle dt  \nonumber\\ 
&\qquad + \varepsilon^2 \left[ \mathrm{Tr}(Q) - \frac{\|Q^\frac{1}{2}\varphi'_{\beta_t}\|^2}{\gamma(\beta_t, v_t^0)} \right ] dt,
\end{align}
where $b>0$ is the constant appearing in Lemma \ref{U decomp}.
\end{thm}

\begin{rem}
\begin{itemize}
\item [(i)] It is worth remarking that the final term in \eqref{ineq} is in fact stabilizing as $t\to\tau_\infty$.
Indeed, by our assumption on the initial condition, we have that $\gamma(\beta_t, v_t^0) >0$ for any $t<\tau_\infty$.  Therefore thanks to the sign of that final term, as $t\uparrow \tau_\infty$ this term converges to $-\infty$.   This is consistent with the fact that $\|z_t\|^2 = \|u_t - \varphi_{\beta_t}\|^2$ does not explode as $t\uparrow \tau_\infty$ even though $\gamma(\beta_t, v_t^0)\downarrow 0$ (see Corollary \ref{non-explosion}).
\item [(ii)] We can also see in the first term in \eqref{ineq} the effect of the exponential decay of the semigroup $(V_\alpha(t))_{t\geq0}$ in the decomposition of $(U_\alpha(t))_{t\geq0}$ (see Lemma \ref{U decomp}).  This occurs precisely because we have chosen the process $(\beta_t)_{t\geq0}$ to be such that $z_t =  u_t - \varphi_{\beta_t}$ is orthogonal to the space spanned by $ \varphi'_{\beta_t}$ (see Proposition \ref{zero}).  The effect of the projection part of $U_{\beta_t}(t)$ on $z_t$ is thus zero for all $t\geq0$.
\end{itemize} 
\end{rem}

Before we prove the theorem, we state a corollary which exploits the exponential decay term in \eqref{ineq}, yielding exponential decay in the limit as $\varepsilon\to0$.

\begin{cor}
\label{cor:ineq}
Suppose that the initial condition $\|z_0\|$ and $\varepsilon>0$ are small enough so that
\begin{equation}
\label{small initial1}
\norm{z_0}^2 + \varepsilon^\frac{1}{2}  + 2b^{-1}\varepsilon^2\mathrm{Tr}(Q) < \frac{\|\varphi'_0\|^2}{2\|\varphi''_0\|} \wedge \frac{b^2}{32c^2},
\end{equation}
where $b$ is the same as in Lemma \ref{U decomp} and $c$ is the Lipschitz constant of $f'$.  Define
\[
\rho_\varepsilon := \inf\left\{t>0: 2\int_{0}^{t} e^{br/2}\langle z_r,B(r)dW^Q_r\rangle \geq \varepsilon^{-\frac{1}{2}}\right\}.
\]
Then $\tau_\infty>\rho_\varepsilon$ and 
\begin{equation}
\label{cor:main}
\norm{z_t}^2 \leq  e^{-\frac{b}{2}t}(\norm{z_0}^2 + \varepsilon^\frac{1}{2})  + 2b^{-1}\varepsilon^2\mathrm{Tr}(Q)(1 - e^{-\frac{b}{2}t})
\end{equation}
for all $t\in[0, \rho_\varepsilon]$.  In particular, since $\rho_\varepsilon\to\infty$ almost surely as $\varepsilon \to 0$, in the limit as $\varepsilon\to0$ we recover the inequality
\begin{align}
\label{cor:second}
\norm{z_t}^2 &\leq  e^{-\frac{b}{2}t}\norm{z_0}^2, \qquad t\geq0.
\end{align}
\end{cor}

\begin{rem}
The inequality \eqref{cor:second} should be compared to the classical results about the stability of traveling waves in the deterministic setting such as \cite[Theorem 1.1, Chapter 5]{volpert-volpert}.  Indeed \eqref{cor:second} agrees exactly with this result, since it says that if the initial condition $u_0$ is such that $\|u_0 - \varphi_{\beta_0}\|$ is small enough, then the solution to the deterministic equation (i.e \eqref{main} with $\varepsilon =0$) will converge exponentially fast towards $\varphi_\alpha$ where $\alpha = \lim_{t\to\infty}\beta_t$. 
\end{rem}

Note that the point of the decomposition in \eqref{zeq} is that the operator $\mathcal{L}_\alpha$ given by \eqref{L alpha} is \textit{linear} (it is in fact the linearization of $\mathcal{D}(A)\ni v\mapsto Av + f(\varphi_\alpha + v)$).  However, we can also consider $(z_t)_{t\in[0, \tau_\infty)}$ as a solution to the stochastic evolution equation given by 
\begin{equation}
\label{zeq2}
dz_t = [Az_t + f(\varphi_{\beta_t} + z_t)  - f(\varphi_{\beta_t})]dt + \varepsilon B(t)dW^Q_t - d\varphi_{\beta_t}.
\end{equation}
From this point of view, we can obtain a similar inequality to that of Theorem \ref{thm: ineq} where we preserve the nonlinearity. This theorem is useful for the long-time results in the following section. We remark that the following result holds without the Assumption \ref{assump A} (ii).
\begin{thm}
\label{thm: ineq better}
Suppose there exists $\omega_0 \in \R$ such that $\norm{P^A_t} \leq \exp(t\omega_0)$ for all $t\geq0$.
For $v\in H$, let 
\begin{equation}
\label{Xi definition}
 \Xi (v) := \limsup_{h\downarrow 0} \frac{1}{h}\left( \|P^A_hv\|^2 - \|v\|^2 \right).
\end{equation}
Then for any $t$ in $[0, \tau_\infty)$, it holds that
\begin{align*}
d\norm{z_t}^2 &\leq  \left[\Xi(z_t)+ 2\langle f(\varphi_{\beta_t} + z_t) - f(\varphi_{\beta_t}), z_t\rangle\right] dt+2\varepsilon\langle z_t,B(t)dW^Q_t\rangle\nonumber \\ 
&\quad + \varepsilon^2 \left[ \mathrm{Tr}(Q) - \frac{\|Q^\frac{1}{2}\varphi'_{\beta_t}\|^2}{\gamma(\beta_t, v_t^0)} \right ] dt.
\end{align*}
\end{thm}

\subsection{Proofs}

In order to prove the results of Section \ref{sec:local stability}, we will need the following lemmas.
\begin{lemma}\label{lem variation of parameters bound}
There exists a constant $\cK$ such that for all $\alpha_1,\alpha_2 \in \R$ and $h\geq 0$,
\[
\norm{U_{\alpha_1 + \alpha_2}(h) - U_{\alpha_2}(h)} \leq \cK \alpha_1 h,
\]
(recall that $(U_\alpha(t))_{t\geq0}$ is the semigroup generated by $\mathcal{L}_\alpha$ given by \eqref{L alpha}).
\end{lemma}
\begin{proof}
Note that $\norm{U_{\alpha_1 + \alpha_2}(h) - U_{\alpha_2}(h)} \leq \norm{U_{\alpha_1}(h) - U_{0}(h)}$, since  $\norm{U_{\alpha}(t)} \leq 1$ for all $\alpha\in\R$. The operator $(\cL_{\alpha_1} - \cL_0) = f'(\varphi_{\alpha_1}) - f'(\varphi_0)$ is bounded over its domain $E$ by assumption. We may therefore use the variation of parameters formula \cite[Page 161]{engel-nagel:01} to write for any $v\in H$
\[
(U_{\alpha_1}(h) - U_0(h))v = \int_0^h U_0(h-r) \left( f'(\varphi_{\alpha_1}) - f'(\varphi_0)\right)U_{\alpha_1}(r)v dr.
\]
The result now follows from the Lipschitz property of $f'$ and $\alpha\mapsto \varphi_\alpha$, as well as the fact that $\norm{U_{\alpha}(t)} \leq 1$ for all $\alpha\in\R$.
\end{proof}

\begin{lemma}\label{lem:G bound}
For $G$ defined by \eqref{G}, it holds that
\[
\|G(z, \beta, \beta)\| \leq \frac{c}{2}\|z\|^2, \quad \forall z\in H, \beta\in\R,
\]
where $c$ is the Lipschitz constant of $f'$ (which is independent of $z$ and $\beta$).
\end{lemma}
\begin{proof}
We first note (by Assumption \ref{assump f} (ii) on $f$) that we may write
\[
f(z +\varphi_{\beta}) - f(\varphi_{\beta}) = \int_0^1f'(\theta z +\varphi_{\beta})zd\theta.
\]
Therefore
\[
G(z, \beta, \beta) = \int_0^1[f'(\theta z +\varphi_{\beta}) - f'(\varphi_{\beta})]zd\theta, 
\]
so that
\[
\|G(z, \beta, \beta)\| \leq \int_0^1\|f'(\theta z +\varphi_{\beta}) - f'(\varphi_{\beta})\|d\theta \|z\| \leq c\int_0^1\theta d\theta \|z\|^2, 
\]
where $c$ is the Lipschitz constant of $f'$.
\end{proof}

We can now prove Theorem \ref{thm: ineq}.

\begin{proof}[Proof of Theorem \ref{thm: ineq}.]
Suppose that $s\leq t \leq T < \tau_\infty$.  We have that the mild solution to \eqref{zeq} is given by 
\begin{equation}
\label{z mild}
z_t = U_\alpha(t-s)z_s + \int_s^tU_\alpha(t - r)G(z_r, \beta_r, \alpha)dr + \varepsilon\int_s^tU_\alpha(t-r)B(r)dW_r^Q - \int_s^tU_\alpha(t-r)d\varphi_{\beta_r}.
\end{equation}
Applying $U_{\alpha}(T-t)$ on both sides of \eqref{z mild} and using the SDE \eqref{beta} governing the behavior of $(\beta_t)_{t\in[0, \tau_\infty)}$, we see that
\begin{align*}
U_{\alpha}(T-t)z_t &= U_{\alpha}(T-s)z_s+ \int_s^t U_{\alpha}(T-r)\kappa_1(r,\beta_r,v^0_r)dW^Q_r\\ 
&\qquad + \int_s^t U_{\alpha}(T-r)\kappa_2(z_r, v^0_r, \beta_r, \alpha)dr,
\end{align*}
where for notational purposes we have set  $\kappa_1(r, \beta_r, v^0_r) := \varepsilon B(r) - \varphi'_{\beta_r}\sigma(r, \beta_r, v^0_r)$ and 
\[
\kappa_2(z_r, v^0_r, \beta_r, \alpha) := {G}(z_r, \beta_r, \alpha) - \varphi'_{\beta_r}\mu(r, \beta_r,v^0_r) - \frac{\varepsilon^2}{2}\frac{\langle \varphi'_{\beta_r},Q\varphi'_{\beta_r}\rangle}{\gamma(\beta_r,v^0_r)^2}\varphi''_{\beta_r}, \quad r\in[0, \tau_\infty),
\]
where we recall that $\sigma(r, \beta_r, v^0_r)$ is defined in \eqref{sigma}.
Let $Y_r = U_{\alpha}(T-r)z_r$ for any $r\in[0, T]$. Then it follows from Ito's Lemma (see \cite[Theorem 4.17]{D-Z}) that
\begin{multline*}
\norm{Y_t}^2 = \norm{Y_s}^2+ 2\int_s^t \langle Y_r,U_{\alpha}(T-r)\kappa_1(r,\beta_r,v^0_r)dW^Q_r\rangle \\ + 2\int_s^t \langle Y_r,U_{\alpha}(T-r)\kappa_2(z_r, v^0_r, \beta_r,\alpha)\rangle dr \\ + \int_s^t  \mathrm{Tr}\left( U_{\alpha}(T-r)\kappa_1(r, \beta_r,v^0_r)Q\kappa_1(r, \beta_r,v^0_r)^*U^*_{\alpha}(T-r)\right)dr.
\end{multline*}
Now taking $T = t$, and choosing $\alpha = \beta_t$ this yields
\begin{multline}
\label{taking T=t}
\norm{z_t}^2 = \norm{U_{\beta_t}(t-s)z_s}^2+ 2\int_s^t \langle U_{\beta_t}(t-r)z_r,U_{\beta_t}(t-r)\kappa_1(r,\beta_r,v^0_r)dW^Q_r\rangle \\ + 2\int_s^t \langle U_{\beta_t}(t-r)z_r,U_{\beta_t}(t-r)\kappa_2(z_r, v^0_r, \beta_r,\beta_t)\rangle dr \\ + \int_s^t  \mathrm{Tr}\left( U_{\beta_t}(t-r)\kappa_1(r, \beta_r,v^0_r)Q\kappa_1(r, \beta_r,v^0_r)^*U^*_{\beta_t}(t-r)\right)dr.
\end{multline}
Now take a partition $(t_k)_{k=0}^M$ of points between $[s,t]$, with $t_k-t_{k-1} = h$ for some $h>0$. Applying the above formula repeatedly, we find that
\begin{multline}
\norm{z_{t}}^2 -\norm{z_{s}}^2= \sum_{k=1}^M \norm{U_{\beta_{t_k}}(h)z_{t_{k-1}}}^2-\norm{z_{t_{k-1}}}^2 \\ 
+ 2\sum_{k=1}^M \int_{t_{k-1}}^{t_k} \left\langle U_{\beta_{t_k}}(t_k-r)z_r,U_{\beta_{t_k}}(t_k-r)\kappa_2(z_r, v^0_r, \beta_r,\beta_{t_k})\right\rangle dr \\ 
+ \sum_{k=1}^M\int_{t_{k-1}}^{t_k}  \mathrm{Tr}\left( U_{\beta_{t_k}}(h)\kappa_1(r, \beta_r,v^0_r)Q\kappa_1(r, \beta_r,v^0_r)^*U^*_{\beta_{t_k}}(h)\right)dr \\
+2\sum_{k=1}^M \int_{t_{k-1}}^{t_k} \langle U_{\beta_{t_k}}(t_k-r)z_r,U_{\beta_{t_k}}(t_k-r)\kappa_1(r,\beta_r,v^0_r)dW^Q_r\rangle.\label{equation zT squared temp 1}
\end{multline}
Note that the potential unboundedness of the generator of $U_\alpha$ makes things a little more difficult. The aim is to deduce from \eqref{equation zT squared temp 1} that 
\begin{align}
\norm{z_{t}}^2 &\leq \norm{z_{s}}^2 - b\int_{s}^{t}\norm{z_r}^2 dr+2\int_{s}^{t} \langle z_r,\kappa_1(r,\beta_r,v^0_r)dW^Q_r\rangle   \nonumber \\ 
&\quad + \int_{s}^{t} \Big[2\left\langle z_r, \kappa_2(z_r, v^0_r, \beta_r,\beta_r) \right\rangle +  \mathrm{Tr}\left( \kappa_1(r, \beta_r,v^0_r)Q\kappa_1(r, \beta_r,v^0_r)^*\right) \Big]dr,\label{z squared b bound 1}
\end{align}
where $b>0$ is as in Lemma \ref{U decomp}.  In order to prove this claim we treat each term in \eqref{equation zT squared temp 1} separately.

\vspace{0.2cm}
\noindent\textit{First term:}  We firstly claim that (noting the dependence of $M$ on $h$)
\begin{equation}
\lim_{h\to 0} \sum_{k=1}^M \left(\norm{U_{\beta_{t_k}}(h)z_{t_{k-1}}}^2- \sum_{k=1}^M \norm{U_{\beta_{t_{k-1}}}(h)z_{t_{k-1}}}^2\right) =0.\label{eqn norms equal 1}
\end{equation}
Indeed, using the reverse triangle inequality, the fact that $\norm{U_\alpha(t)}\leq 1$ and Lemma \ref{lem variation of parameters bound}, by setting  $\bar{\cK} = \sup_{r\in [s,t]}\norm{z_r}$ we see that
\begin{align*}
\left| \sum_{k=1}^M \norm{U_{\beta_{t_k}}(h)z_{t_{k-1}}}^2 -  \norm{U_{\beta_{t_{k-1}}}(h)z_{t_{k-1}}}^2 \right| &\leq 2\bar{\cK} \sum_{k=1}^M \norm{(U_{\beta_{t_k}}(h)-U_{\beta_{t_{k-1}}}(h))z_{t_{k-1}}}\\  
&\leq 2(t-s)\cK\bar{\cK}\sup_{r_1,r_2\in [s,t]: |r_1 - r_2|\leq h}|\beta_{r_2} - \beta_{r_1}|,
\end{align*} 
which converges to $0$ as $h\to0$ by the continuity of $(\beta_r)_{r\in[0, \tau_\infty)}$.  Therefore
\begin{align*}
&\limsup_{h\to 0} \sum_{k=1}^M \left[\norm{U_{\beta_{t_k}}(h)z_{t_{k-1}}}^2 - \|z_{t_{k-1}}\|^2\right] = \limsup_{h\to 0} \sum_{k=1}^M \left[\norm{U_{\beta_{t_{k-1}}}(h)z_{t_{k-1}}}^2 - \|z_{t_{k-1}}\|^2\right] \\
&\qquad\qquad\qquad\qquad\qquad\leq \lim_{h\to 0} \sum_{k=1}^M \left[e^{-bh} - 1\right]\|z_{t_{k-1}}\|^2 \leq -b\int_s^t\|z_{r}\|^2dr,
\end{align*}
%\mac{I think that the above three each need to be $lim-sup$. Remember that we cannot assume a priori that the first two limits exist} \ing{agreed - I have changed them}
where the second line follows from Lemma \ref{U decomp} and the fact that by Proposition \ref{zero} $\langle z_{r}, \varphi_{\beta_{r}}\rangle =0$ for all $r\in[0, \tau_\infty)$.

\vspace{0.2cm}
\noindent\textit{Second term:} 
We have that
\begin{align*}
&2\sum_{k=1}^M \int_{t_{k-1}}^{t_k} \left\langle U_{\beta_{t_k}}(t_k-r)z_r,U_{\beta_{t_k}}(t_k-r)\kappa_2(z_r, v^0_r, \beta_r,\beta_{t_k})\right\rangle dr  \\
&= 2\int_{s}^{t} \left\langle U_{\beta_{k(r)}}(k(r)-r)z_r,U_{\beta_{k(r)}}(k(r)-r)\kappa_2(z_r, v^0_r, \beta_r,\beta_{k(r)})\right\rangle dr
\end{align*}
where $k(r) := t_k$ if $r\in(t_{k-1}, t_k]$ for $k\in\{1, \dots, M\}$.  Since it holds that $\|U_{\beta_{k(r)}}(k(r) - r)v - v\| \to 0$ as $h\to0$ for any $v\in H$ and $r\in(\tau_{k-1}, \tau_k]$ by Lemma \ref{lem variation of parameters bound}), we see that by the dominated convergence theorem 
\begin{align*}
&2\sum_{k=1}^M \int_{t_{k-1}}^{t_k} \left\langle U_{\beta_{t_k}}(t_k-r)z_r,U_{\beta_{t_k}}(t_k-r)\kappa_2(z_r,  v^0_r, \beta_r,\beta_{t_k})\right\rangle dr  \to 2\int_{s}^{t} \left\langle z_r,\kappa_2(z_r, v^0_r, \beta_r,\beta_r)\right\rangle dr,
\end{align*}
as $h\to0$.

\vspace{0.2cm}
\noindent\textit{Third term:} Similarly to the second term, we have
\begin{align*}
&\sum_{k=1}^M\int_{t_{k-1}}^{t_k}  \mathrm{Tr}\left( U_{\beta_{t_k}}(h)\kappa_1(r, \beta_r,v^0_r)Q\kappa_1(r, \beta_r,v^0_r)^*U^*_{\beta_{t_k}}(h)\right)dr \\
&\quad \to  \int_{s}^{t}  \mathrm{Tr}\left(\kappa_1(r, \beta_r,v^0_r)Q\kappa_1(r, \beta_r,v^0_r)^*\right)dr
\end{align*}
as $h\to0$.

\vspace{0.2cm}
\noindent\textit{Fourth term:}
For the final term in \eqref{equation zT squared temp 1}, observe that
\begin{align*}
&2\sum_{k=1}^M \int_{t_{k-1}}^{t_k} \langle U_{\beta_{t_k}}(t_k-r)z_r,U_{\beta_{t_k}}(t_k-r)\kappa_1(r,\beta_r,v^0_r)dW^Q_r\rangle \\
&\quad= 2 \int_{s}^{t} \langle z_r,\kappa_1(r, \beta_r,v^0_r)dW^Q_r\rangle + 2 \int_{s}^{t} \Big\langle J(k(r), r)z_r,\kappa_1(r,\beta_r,v^0_r)dW^Q_s\Big\rangle,
\end{align*}
where $k(r)$ is defined as in the bound for the second term above, and for notational purposes we have set
$J(k(r), r) := U^*_{\beta_{k(r)}}(k(r)-r)U_{\beta_{k(r)}}(k(r)-r)-I$.
Now
\begin{multline}
\E\left[\left( \int_{s}^{t} \Big\langle J(k(r), r)z_r,\kappa_1(r,\beta_r,v^0_r)dW^Q_r\Big\rangle\right)^2 \right] \\
= \int_{s}^{t} \Big\langle J(k(r), r)z_r, \kappa_1(r,\beta_r,v^0_r)Q\kappa_1^*(r,\beta_r,v^0_r)J(k(r), r)z_r \Big\rangle dr.
\end{multline}
This goes to zero as $h\to 0$ through the dominated convergence theorem, so that we conclude that
\begin{align*}
&2\sum_{k=1}^M \int_{t_{k-1}}^{t_k} \langle U_{\beta_{t_k}}(t_k-r)z_r,U_{\beta_{t_k}}(t_k-r)\kappa_1(r,\beta_r,v^0_r)dW^Q_r\rangle \to 2 \int_{s}^{t} \langle z_r,\kappa_1(r, \beta_r,v^0_r)dW^Q_r\rangle,
\end{align*}
almost surely as $h\to 0$.

\vspace{0.2cm}
\noindent\textit{Conclusion:}  Using the above calculations, we can thus see that by taking the limit as $h\to0$ in \eqref{equation zT squared temp 1}, \eqref{z squared b bound 1} holds almost surely.  It remains to deduce the required inequality from \eqref{z squared b bound 1}.

Firstly we can note that since $\langle z_{r}, \varphi'_{\beta_{r}}\rangle =0$ for all $r\in[0, \tau_\infty)$ we have by definition of $\kappa_1$ and $\kappa_2$ that
\begin{align}
\label{ineq conc1}
 \langle z_r,\kappa_1(r,\beta_r,v^0_r)dW^Q_r\rangle = \varepsilon\langle z_r,B(r)dW^Q_r\rangle,
\end{align}
and
\begin{align}
\label{ineq conc2}
\langle z_r, \kappa_2(z_r, v^0_r, \beta_r,\beta_r) \rangle = \langle z_r, G(z_r,\beta_r,\beta_r) \rangle - \frac{\varepsilon^2}{2}\frac{\langle \varphi'_{\beta_r},Q\varphi'_{\beta_r}\rangle}{\gamma(\beta_r,v^0_r)^2}\langle z_r,\varphi''_{\beta_r}\rangle.
\end{align}
Moreover, we can calculate (using the assumption that $B^*(r)B(r) = \mathbf{Id}$)
%\mac{ and using an orthonormal basis containing $B^*(r)\varphi'_{\beta_r}\norm{\varphi'}^{-1}$ to evaluate the trace} \ing{I am not sure we need to use this ON basis - it should be true any ON basis?})
\begin{align}
\label{ineq conc3}
&{\rm Tr}\left( \kappa_1(r,\beta_r,v^0_r)Q\kappa_1(r,\beta_r,v^0_r)^*\right)  
& =\varepsilon^2\mathrm{Tr}(Q) + \varepsilon^2\frac{\langle Q\varphi'_{\beta_r} , \varphi'_{\beta_r}\rangle}{\gamma(\beta_r, v_r^0)^2}\left(\|\varphi'_0\|^2 - 2 \gamma(\beta_r, v_r^0)\right).
\end{align}
Substituting these three observations into \eqref{z squared b bound 1} then yields the result.
\end{proof}

\begin{proof}[Proof of Corollary \ref{cor:ineq}]
By a simple application of It\^o's formula to $e^{bt/2}\|z_t\|^2$, thanks to Theorem \ref{thm: ineq} for any $t<\tau_\infty$, we have
\begin{align*}
\norm{z_t}^2 &\leq e^{-\frac{b}{2}t}\norm{z_0}^2 - \frac{b}{2}\int_0^te^{-\frac{b}{2}(t-r)}\|z_r\|^2dr + 2\varepsilon\int_{0}^{t} e^{-\frac{b}{2}(t-r)}\langle z_r,B(r)dW^Q_r\rangle   \\
&\qquad + 2\int_{0}^{t}e^{-\frac{b}{2}(t-r)}\left\langle z_r, G(z_r, \beta_r,\beta_r) \right\rangle dr + 2b^{-1}\varepsilon^2\mathrm{Tr}(Q)(1 - e^{-\frac{b}{2}t}),
\end{align*}
Thus by Lemma \ref{lem:G bound},
and by the definition of $\rho_\varepsilon$, it follows that for $t<\tau_\infty\wedge\rho_\varepsilon$
\begin{align}
\label{z bound 1}
\norm{z_t}^2 &\leq e^{-\frac{b}{2}t}\norm{z_0}^2 - \frac{b}{2}\int_0^te^{-\frac{b}{2}(t-r)}\|z_r\|^2dr + \varepsilon^\frac{1}{2} e^{-\frac{b}{2}t} \nonumber\\
& \qquad + c\int_{0}^{t}e^{-\frac{b}{2}(t-r)}\|z_r\|^3 dr  + 2b^{-1}\varepsilon^2\mathrm{Tr}(Q)(1 - e^{-\frac{b}{2}t}).
\end{align}
Define $\tilde{\rho} := \inf\{t>0: \|z_t\| \geq {b}/{4c}\}$, so that $\tilde{\rho} >0$ a.s. by our assumption \eqref{small initial1}.
Then for 
$t<\tau_\infty\wedge\rho_\varepsilon\wedge \tilde{\rho}$ it holds that
\begin{align*}
\norm{z_t}^2 &\leq e^{-\frac{b}{2}t}\norm{z_0}^2 + \varepsilon^\frac{1}{2} e^{-\frac{b}{2}t} + \frac{1}{2}\sup_{r\leq t}\|z_r\|^2(1 - e^{-\frac{b}{2}t})   + 2b^{-1}\varepsilon^2\mathrm{Tr}(Q)(1 - e^{-\frac{b}{2}t}).
\end{align*}
This implies that 
\begin{equation}
\label{sup bound z}
\sup_{r\leq t}\|z_r\|^2 \leq 2\left(\norm{z_0}^2 + \varepsilon^\frac{1}{2}  + 2b^{-1}\varepsilon^2\mathrm{Tr}(Q)\right) <  \frac{b^2}{16c^2}
\end{equation}
for all $t<\tau_\infty\wedge\rho_\varepsilon\wedge \tilde{\rho}$ by the assumption $\eqref{small initial1}$.  
Then we must have that $\tau_\infty\wedge\rho_\varepsilon < \tilde{\rho}$, so that \eqref{sup bound z} holds for all $t<\tau_\infty\wedge\rho_\varepsilon $. Returning to \eqref{z bound 1}, we thus see that
\begin{align*}
\norm{z_t}^2 &\leq e^{-\frac{b}{2}t}\norm{z_0}^2 - \frac{b}{4}\int_0^te^{-\frac{b}{2}(t-r)}\|z_r\|^2dr + \varepsilon^\frac{1}{2} e^{-\frac{b}{2}t}  + 2b^{-1}\varepsilon^2\mathrm{Tr}(Q)(1 - e^{-\frac{b}{2}t})\\
&\leq  e^{-\frac{b}{2}t}(\norm{z_0}^2 + \varepsilon^\frac{1}{2})  + 2b^{-1}\varepsilon^2\mathrm{Tr}(Q)(1 - e^{-\frac{b}{2}t}),
\end{align*}
for all $t<\tau_\infty\wedge\rho_\varepsilon $.

Finally, again by the assumption \eqref{small initial1}, it follows that $\norm{z_t}^2 \leq \|\varphi'_0\|^2/2\|\varphi''_0\|$
for all $t<\tau_\infty\wedge\rho_\varepsilon$.  The point is then that on $[0, \tau_\infty\wedge\rho_\varepsilon)$ we have by definition of $\gamma$ (see \eqref{gamma}) that
\begin{align*}
\gamma(\beta_t, v^0_t) &= -\langle u_t, \varphi''_{\beta_t}\rangle = \|\varphi'_0\|^2 -\langle z_t, \varphi''_{\beta_t}\rangle \geq \|\varphi'_0\|^2 -\|z_t\|\|\varphi''_{0}\| \geq \frac{\|\varphi'_0\|^2}{2},
\end{align*}
so that $\tau_\infty > \rho_\varepsilon$, recalling that by definition $\tau_\infty$ is the first time that $\gamma(\beta_t, v^0_t)=0$.

In conclusion, we have that under the assumption \eqref{small initial1} it holds that  $\tau_\infty > \rho_\varepsilon$ and
\begin{align*}
\norm{z_t}^2 &\leq  e^{-\frac{b}{2}t}(\norm{z_0}^2 + \varepsilon^\frac{1}{2})  + 2b^{-1}\varepsilon^2\mathrm{Tr}(Q)(1 - e^{-\frac{b}{2}t})
\end{align*}
for all $t\in[0, \rho_\varepsilon]$.
\end{proof}

We can finally prove  Theorem \ref{thm: ineq better}.

\begin{proof}[Proof of Theorem \ref{thm: ineq better}.]
The proof is very similar to that of Theorem \ref{thm: ineq} but this time we consider $(z_t)_{t\in[0, T]}$ for $T<\tau_\infty$ as a mild solution to \eqref{zeq2} i.e.
\begin{equation}
\label{z mild better}
z_t = P^A_{t-s}z_s + \int_s^tP^A_{t-r}\left(f(z_r + \varphi_{\beta_r}) -f(\varphi_{\beta_r})\right)dr + \varepsilon\int_s^tP^A_{t-r}B(r)dW_r^Q - \int_s^tP^A_{t-r}d\varphi_{\beta_r},
\end{equation}
for all $s\leq t \leq T$.  
In a very similar way to the derivation of \eqref{taking T=t} in the proof of Theorem \ref{thm: ineq}, we see that
\begin{multline*}
\norm{z_t}^2 = \norm{P^A_{t -s}z_s}^2+ 2\int_s^t \langle P^A_{t -r}z_r,P^A_{t -r}\left(f(z_r + \varphi_{\beta_r}) -f(\varphi_{\beta_r})\right)  dr\\
+ 2\int_s^t \langle P^A_{t -r}z_r,P^A_{t -r}\kappa_1(r,\beta_r,v^0_r)dW^Q_r\rangle + 2\int_s^t \langle P^A_{t -r}z_r,P^A_{t -r}\kappa_3(z_r, v^0_r, \beta_r)\rangle dr \\ + \int_s^t  \mathrm{Tr}\left( P^A_{t -r}\kappa_1(r, \beta_r,v^0_r)Q\kappa_1(r, \beta_r,v^0_r)^*(P^A_{t -r})^*\right)dr.
\end{multline*}
where $\kappa_1(r, \beta_r, v^0_r) := \varepsilon B(r) - \varphi'_{\beta_r}\sigma(r, \beta_r, v^0_r)$ as in the proof of Theorem \ref{thm: ineq} and 
\[
\kappa_3(z_r, v^0_r, \beta_r) := - \mu(r, \beta_r,v^0_r)\varphi'_{\beta_r} - \frac{\varepsilon^2}{2}\frac{\langle \varphi'_{\beta_r},Q\varphi'_{\beta_r}\rangle}{\gamma(\beta_r,v^0_r)^2}\varphi''_{\beta_r}.
\]
Again take a partition $(t_k)_{k=0}^M$ of points between $[s,t]$, with $t_k-t_{k-1} = h$ for some $h>0$. Applying the above formula repeatedly, we find that
\begin{multline}
\norm{z_{t}}^2 -\norm{z_{s}}^2 \\
= \sum_{k=1}^M \left(\norm{P^A_{h}z_{t_{k-1}}}^2-\norm{z_{t_{k-1}}}^2+ 2\int_{t_{k-1}}^{t_k} \left\langle P^A_{t_k-r}z_r,P^A_{t_k-r}\left(f(z_r + \varphi_{\beta_r}) -f(\varphi_{\beta_r})\right)\right\rangle dr\right) \\
+ \sum_{k=1}^M \int_{t_{k-1}}^{t_k} \left(2\left\langle P^A_{t_k-r}z_r,P^A_{t_k-r}\kappa_3(z_r, v^0_r, \beta_r)\right\rangle +  \mathrm{Tr}\left( P^A_h\kappa_1(r, \beta_r,v^0_r)Q\kappa_1(r, \beta_r,v^0_r)^*(P^A_h)^*\right)\right)dr \\
+2\sum_{k=1}^M \int_{t_{k-1}}^{t_k} \langle P^A_{t_k-r}z_r,P^A_{t_k-r}\kappa_1(r,\beta_r,v^0_r)dW^Q_r\rangle.\label{equation zT squared temp 1 better}
\end{multline}
Once again the aim is to take the $\limsup$ as $h\to0$ in the above.  The second, third and fourth terms are dealt with in exactly the same way as in the proof of Theorem \ref{thm: ineq}, so it suffices to concentrate on the first term. 

To this end note that
\begin{align*}
&\sum_{k=1}^M \left(\norm{P^A_{h}z_{t_{k-1}}}^2-\norm{z_{t_{k-1}}}^2+ 2\int_{t_{k-1}}^{t_k} \left\langle P^A_{t_k-r}z_r,P^A_{t_k-r}\left(f(z_r + \varphi_{\beta_r}) -f(\varphi_{\beta_r})\right)\right\rangle dr\right) \\
 &= \sum_{k=1}^M \left(\norm{P^A_{h}z_{t_{k-1}}}^2-\norm{z_{t_{k-1}}}^2\right)+ 2\int_{s}^{t} \left\langle P^A_{k(r)-r}z_r,P^A_{k(r)-r}\left(f(z_r + \varphi_{\beta_r}) -f(\varphi_{\beta_r})\right)\right\rangle dr
\end{align*}
where $k(r) := t_k$ if $r\in(t_{k-1}, t_k]$ for $k\in\{1, \dots, M\}$. By the assumption in the theorem that there exists $\omega_0\in\R$ such that $\norm{P^A_t} \leq \exp(t\omega_0)$ for all $t\geq0$, it follows that $\|P^A_hz_{t_{k-1}}\|^2 - \|z_{t_{k-1}}\|^2 \leq [\exp(2\omega_0h) - 1]\|z_{t_{k-1}}\|^2$.  Combining this observation with the reverse Fatou lemma, we see that
\[
\limsup_{h\to 0} \sum_{k=1}^M \left(\norm{P^A_{h}z_{t_{k-1}}}^2-\norm{z_{t_{k-1}}}^2\right) \leq \int_0^T\Xi(z_r)dr.
\]
%\mac{I am not sure that I understand how this works. Remember that it is not clear that $r \mapsto \Xi(z_r)$ is continuous in $r$. We would need to show that
%\[
%\limsup_{s\to t,h\to 0}h^{-1}(\norm{z_{s+h}}^2 - \norm{z_s}^2) \leq \limsup_{h\to 0}h^{-1}(\norm{z_{t+h}}^2 - \norm{z_t}^2)  .
%\]
%We could just defined $\Xi$ to be the above expression on the right, but of course it would then not be a function $H \to \R$.}
%
%\mac{I suggest that we assume at the start of the theorem (and also for the corresponding theorem in Section 7) that there exists an $\omega_0 \in \R$ such that,
%\[
%\norm{P^A_t} \leq \exp(t\omega_0).
%\]
%This is termed a `quasi-contractive' semigroup in Engel and Nagel. The heat semigroup satisfies this property. Or we could make this an assumption at the start of the paper, as it is strictly more general than the spectral gap assumption - i.e. any $A$ satisfying the spectral gap property must be quasicontractive. We may then apply the Fatou Lemma as outlined above.
%}
%\ing{OK, I agree.  I think it is probably better to put the extra assumption in the Theorem as you suggest.}
The dominated convergence theorem also implies that as $h\to 0$,
\[
\int_{s}^{t} \left\langle P^A_{k(r)-r}z_r,P^A_{k(r)-r}\left(f(z_r + \varphi_{\beta_r}) -f(\varphi_{\beta_r})\right)\right\rangle dr
\to \int_{s}^{t} \left\langle z_r,f(z_r + \varphi_{\beta_r}) -f(\varphi_{\beta_r})\right\rangle dr.
\]
%Since $\int_0^t \norm{z}_r^2 dr < \infty$,
%Then by assumption \eqref{assumption}, we have that for $h$ small enough
%\begin{align*}
%&\sum_{k=1}^M \left(\norm{P^A_{h}z_{t_{k-1}}}^2-\norm{z_{t_{k-1}}}^2+ 2\int_{t_{k-1}}^{t_k} \left\langle P^A_{t_k-r}z_r,P^A_{t_k-r}\left(f(z_r + \varphi_{\beta_r}) -f(\varphi_{\beta_r})\right)\right\rangle dr\right) \\
%&\leq \sum_{k=1}^Mh\Xi(\|z_{t_{k-1}}\|) - 2\sum_{k=1}^Mh\left\langle z_{t_{k-1}},\left(f(z_{t_{k-1}} + \varphi_{\beta_{t_{k-1}}}) -f(\varphi_{\beta_{t_{k-1}}})\right)\right\rangle  \\
% &\qquad + 2\int_{s}^{t} \left\langle P^A_{k(r)-r}z_r,P^A_{k(r)-r}\left(f(z_r + \varphi_{\beta_r}) -f(\varphi_{\beta_r})\right)\right\rangle dr.
%\end{align*}
%Since $(P^A_t)_{t\geq0}$ is strongly continuous, thanks to the dominated convergence theorem, we see that by taking the limit as $h\to0$ in the above that
%\begin{align*}
%&\lim_{h\downarrow 0}\sum_{k=1}^M \left(\norm{P^A_{h}z_{t_{k-1}}}^2-\norm{z_{t_{k-1}}}^2+ 2\int_{t_{k-1}}^{t_k} \left\langle P^A_{t_k-r}z_r,P^A_{t_k-r}\left(f(z_r + \varphi_{\beta_r}) -f(\varphi_{\beta_r})\right)\right\rangle dr\right) \\
% &\qquad\qquad\leq  \int_s^t\Xi(\|z_r\|)dr.
%\end{align*}
With this in hand, together with the limits calculated for the second, third and fourth terms of \eqref{equation zT squared temp 1 better} in the proof Theorem \ref{thm: ineq}, we see that taking the $\limsup$ as $h\to0$ in \eqref{equation zT squared temp 1 better} yields 
\begin{align}
\norm{z_{t}}^2 &\leq \norm{z_{s}}^2 + \int_{s}^{t}\Xi(\norm{z_r}) dr +2\int_{s}^{t} \langle z_r,\kappa_1(r,\beta_r,v^0_r)dW^Q_r\rangle   \nonumber \\ 
&\quad +2\int_{s}^{t} \left\langle z_r,f(z_r + \varphi_{\beta_r}) -f(\varphi_{\beta_r})\right\rangle dr\nonumber\\
&\qquad + \int_{s}^{t} \Big[2\left\langle z_r, \kappa_3(z_r,  v^0_r, \beta_r) \right\rangle +  \mathrm{Tr}\left( \kappa_1(r, \beta_r,v^0_r)Q\kappa_1(r, \beta_r,v^0_r)^*\right) \Big]dr \label{z squared b bound 1 better}.
\end{align}
Moreover, we can then use \eqref{ineq conc1}, \eqref{ineq conc3} and the definition of $\kappa_3$ to conclude.
\end{proof}

%%%%%%%%%%%%%%%%%
%Long time behavior
%%%%%%%%%%%%%%%%%

\section{Long-time behavior}
\label{sec:long-time stability}

Again suppose that $(\varphi_\alpha)_{\alpha\in\R}$, $f$ and $A$ satisfy Assumptions \ref{assump varphi}, \ref{assump f} and \ref{assump A} (i) respectively, and that $(u_t)_{t\geq0} = (v^0_t + \varphi_0)_{t\geq0}$ is the solution to \eqref{main} with (deterministic) initial condition $u_0$ such that $u_0 - \varphi_0\in H$ according to Proposition \ref{K-S}.

Let $t\mapsto \beta^*_t$ be any function on $[0, \infty)$ such that for all $t\geq0$, $\beta^*_t$ is a global minimum of the map $\R\ni\alpha\mapsto m(t, \alpha) = \norm{u_t - \varphi_{\alpha}}^2$.  Note that $\beta^*_t$ exists for all $t\geq0$ by Lemma \ref{beta existence} but it may not be unique. Define
\[
z^*_t := u_t - \varphi_{\beta_t^*}, \quad t\geq0.
\]

The main result of this section is Theorem \ref{Proposition Nonlinear Inequality}, which generalizes the inequality of Theorem \ref{thm: ineq better} to arbitrary time. This theorem is a first step in the long-time analysis of the system. The global stability results of \cite{chen1997existence} lends one hope that, for some traveling wave systems, we might be able to get some sort of long-time bound on $\norm{z^*_t}^2$. In particular, one may infer from \cite[Theorem 3.1]{chen1997existence} that, under some technical assumptions, if $\varepsilon=0$ (i.e. there is no stochastic term), $u_0\in E$ and $u_0$ is continuous, then $\norm{u_t - \varphi_{\beta_t^*}}_\infty\to 0$ (in supremum norm) as $t\to \infty$. Coming back to our stochastic setting with $\varepsilon >0$, this motivates us to wonder if the stabilizing effect of the internal dynamics of the deterministic system could balance the disorder coming from the noise.  In such a case then a long-time bound on $\norm{z^*_t}^2$ might be possible. Unfortunately \cite{chen1997existence} uses the method of comparison of ODE's, and the bounds are not easy to adapt to our semigroup formalism. Nevertheless, the development, in future work, of some bounds on the drift term in Theorem \ref{Proposition Nonlinear Inequality} could facilitate for example a long-time bound on the growth of $\E\left[\norm{z^*_t}^2\right]$ (see also Remark \ref{remark long time application}).

\begin{assump}\label{Assmpt Long Time 1}
For $t \geq 0$, let $Q_t := \int_0^t B(s)^*(P^A_{t-s})^*QP^A_{t-s}B(s)ds$. Assume that for all $t > 0$ and for all $\alpha,\beta \in \R$, $\alpha\neq \beta$,
\begin{equation}
\left\langle \varphi_\alpha - \varphi_{\beta},Q_t(\varphi_\alpha - \varphi_\beta)\right\rangle \neq 0.
\end{equation}
Note that a sufficient condition for this to hold is that $Q$ is strictly positive.
\end{assump}
\begin{assump}\label{Assmpt Long Time 2}
For $\alpha\in\R$ and $t > 0$, define
\[
K^\varphi(\alpha,t) := \left(
\begin{array}{ccc}
\langle \varphi_\alpha',Q_t\varphi_\alpha'\rangle & \langle \varphi_\alpha'',Q_t\varphi_\alpha'\rangle &  \langle \varphi_\alpha''',Q_t\varphi_\alpha'\rangle\\
\langle \varphi_\alpha'',Q_t\varphi_\alpha'\rangle & \langle \varphi_\alpha'',Q_t\varphi_\alpha''\rangle & \langle \varphi_{\alpha}'',Q_t\varphi_\alpha'''\rangle\\
\langle \varphi_\alpha''',Q_t\varphi_\alpha'\rangle & \langle \varphi_{\alpha}'',Q_t\varphi_\alpha'''\rangle &  \langle \varphi_\alpha''',Q_t\varphi_\alpha'''\rangle
\end{array}\right).
\]
%define $K^\varphi(\alpha,t)$ to be the $3\times 3$ real symmetric matrix with the following entries
%\begin{multline}
%K^\varphi_{1,1}(\alpha,t) = \langle \varphi_\alpha',Q_t\varphi_\alpha'\rangle \;\;,\;\; K^\varphi_{2,2}(\alpha,t) = \langle \varphi_\alpha'',Q_t\varphi_\alpha''\rangle \;\;,\;\; K^\varphi_{3,3}(\alpha,t) = \langle \varphi_\alpha''',Q_t\varphi_\alpha'''\rangle \\
%K^\varphi_{2,1}(\alpha,t) = K^\varphi_{1,2}(\alpha,t) = \langle \varphi_\alpha'',Q_t\varphi_\alpha'\rangle \;\;,\;\; K^\varphi_{3,1}(\alpha,t) = K^\varphi_{1,3}(\alpha,t) = \langle \varphi_\alpha''',Q_t\varphi_\alpha'\rangle \\
%K^\varphi_{2,3}(\alpha,t) = K^{\varphi}_{3,2}(\alpha,t) = \langle \varphi_{\alpha}'',Q_t\varphi_\alpha'''\rangle .
%\end{multline}
Let $O^\varphi(\alpha,t)$ be an orthonormal matrix and $\Lambda^{\varphi}(\alpha,t)$ a diagonal matrix with diagonal entries $(\lambda^{\varphi}_1(\alpha,t),\lambda^\varphi_2(\alpha,t),\lambda^\varphi_3(\alpha,t))$ such that
\begin{equation}\label{eqn O definition}
K^\varphi(\alpha,t) = O^\varphi(\alpha,t)^T\Lambda^\varphi(\alpha,t)O^\varphi(\alpha,t).
\end{equation}
We choose $O^\varphi(\alpha,t)$ and $\Lambda^\varphi(\alpha,t)$ to be continuous in $\alpha$ (for each $t>0$). 
Assume that for each $\alpha\in \R$ and $t>0$, no more than one of $(\lambda^\varphi_i(\alpha, t))_{i=1,2,3}$ is zero. Assume also that $\varphi''''_\alpha\in H$ (the derivative w.r.t. $\alpha$) exists everywhere and its norm is uniformly bounded. 
\end{assump}
We recall the definition of $\Xi(v)$ for $v\in H$ in Theorem \ref{thm: ineq better} as the map $\Xi (v) := \limsup_{h\downarrow 0} \frac{1}{h}\left( \|P^A_hv\|^2 - \|v\|^2 \right)$.
%\mac{As per my comments in Section 6, I think that we need an extra assumption on $P^A$ for this theorem to hold.}

\begin{thm}\label{Proposition Nonlinear Inequality}
Suppose Assumptions \ref{Assmpt Long Time 1} and \ref{Assmpt Long Time 2} are satisfied, and that there exists $\omega_0 \in \R$ such that $\norm{P^A_t} \leq \exp(t\omega_0)$ for all $t\geq0$. Then for all $t>0$, $\beta^*_t$ is almost surely unique. 
Furthermore for any $0\leq s<t <\infty $ it holds almost surely that
\begin{align*}
\norm{z^*_t}^2 &\leq \norm{z^*_s}^2 + \int_s^t\Big(\Xi(z^*_r)+ 2\langle f(\varphi_{\beta_r^*} + z^*_r) - f(\varphi_{\beta_r^*}), z^*_r\rangle +\varepsilon^2 \Big[ \mathrm{Tr}(Q) - \frac{\|Q^\frac{1}{2}\varphi'_{\beta^*_r}\|^2}{\gamma(\beta^*_r, v_r^0)} \Big ] \Big)dr\\
&\qquad +2\varepsilon\int_s^t\langle z^*_r,B(r)dW^Q_r\rangle.
\end{align*}
\end{thm}

\begin{rem}
The inequality in Theorem \ref{Proposition Nonlinear Inequality} holds despite the fact that $z^*_t$ and $\beta^*_t$ may not be continuous in time. If $z^*_t \in \mathcal{D}(A)$, then $\Xi(z^*_t) = 2\langle z^*_t, Az^*_t\rangle$. If this holds for all $t\geq0$, the inequality in Theorem \ref{Proposition Nonlinear Inequality} becomes an equality.
\end{rem}

\begin{rem}
Assumptions \ref{Assmpt Long Time 1} and \ref{Assmpt Long Time 2} are used to ensure that $|\mathcal{D}_{\delta,T}|\to 0$ as $\delta\to 0$ for any $T\geq0$, where $\mathcal{D}_{\delta,T}$ is defined in the course of the proof.  This proof is given in Lemma \ref{Lemma D delta}, and demonstrates that if the noise is uncorrelated at any two distinct points in space, then through the Girsanov theorem $(u_t)_{t\geq0}$ will also be uncorrelated. We think that this is by no means necessary for $|\mathcal{D}_{\delta,T}|\to 0$ as $\delta\to 0$. In fact, it is possible that even if the noise is quite degenerate, the dynamics of $A$ and $f$ might ensure that $(u_t)_{t\geq0}$ is not. 
\end{rem}

\begin{rem}\label{remark long time application}
Suppose that there were to exist constants $b,C> 0$ such that, for all $u \in E$, 
\[
\Xi(u-\varphi_\alpha)+ 2\langle f(u) - f(u - \varphi_\alpha), u-\varphi_\alpha\rangle \leq -b\norm{u-\varphi_\alpha}^2 + C,
\]
where $\alpha$ is a global minimizer of $\alpha\mapsto \norm{u-\varphi_\alpha}^2$.
Then a consequence of Theorem \ref{Proposition Nonlinear Inequality} would be that
\begin{equation*}
\E\norm{z^*_t }^2 \leq (1 - e^{-bt})C + e^{-bt}\E\norm{z^*_0}^2  + \varepsilon^2\mathrm{Tr}(Q)(1 - e^{-bt}), \quad \forall t\geq0.
\end{equation*}
That is, we would obtain a bound on $\E\norm{z^*_t }^2$ which holds uniformly for all $t>0$. Unfortunately, at the moment we do not have any examples where the above inequality holds. However, we believe that it might be possible for some traveling waves, particularly if we work in a Hilbert space with weighted inner product, and plan to investigate this in the future.
\end{rem}

\subsection{Proof of Theorem \ref{Proposition Nonlinear Inequality} and Lemma \ref{Lemma D delta}}
In order to prove Theorem \ref{Proposition Nonlinear Inequality}, we introduce the following definitions. 

\vspace{0.3cm}
\noindent\textit{The set $E_\delta$:} For $\delta\in(0, 1)$ define $E_\delta \subset E$ by $u\in E_\delta$ $\Leftrightarrow$ i) $u\in E$, ii) $\exists$ a unique global minimum $\Gamma(u)$ of $\alpha\mapsto\norm{u - \varphi_\alpha}^2$, and iii) for all $\alpha \in [\Gamma(u)-\delta,\Gamma(u)+\delta]$, $\gamma(\alpha,u)>\delta$, where we recall that $\gamma(\cdot, u)$ is the `curvature' of the map $\alpha\mapsto\norm{u - \varphi_\alpha}^2$ given by \eqref{gamma}.

\vspace{0.3cm}
\noindent\textit{The set $E_\delta^M$:} For $M>0$ and $\delta\in(0,1)$, let $E_\delta^M\subset E_\delta$ be such that  $u\in E^M_\delta$ $\Leftrightarrow$ i) $u\in E_\delta$, ii) $\norm{u - \varphi_0} < M$, and iii) for all $\alpha\in\R\backslash[\Gamma(u)-\delta,\Gamma(u)+\delta]$,
\[
\norm{u - \varphi_{\alpha}}> \norm{u - \phi_{\Gamma(u)}}  +  \delta^3(\delta\norm{\varphi_0'}+ 2M)^{-1}. 
\]

\vspace{0.3cm}
\noindent\textit{The stopping time $\rho_T^{\delta}$:} For $T>0$, we define $\rho_T^\delta := \inf\Big\{ t\leq T: \norm{v^0_t} \geq\delta^{-1}\Big\}$, with $\rho^\delta_T =T$ if the set is empty.

\vspace{0.3cm}
\noindent\textit{The process $(\eta^\delta_t)_{t\in[0, T]}$:} 
For $T>0$, we now introduce the process $(\eta^\delta_t)_{t\in[0, T]}$, for any $\delta\in(0, \gamma(\beta^*_0, u_0 - \varphi_0))$ in the following recursive way. Let $\tau^{0}= \inf\{ t\geq 0: u_t\in E^{\delta^{-1}}_{\delta} \}\wedge\rho_T^\delta$,
and for any $k\geq 0$, let 
\begin{align*}
\tau^{2k+1}&= \inf\Big\{ t\geq \tau^{2k}: u_t\notin E^{\delta^{-1}}_{\delta/2} \Big\}\wedge\rho_T^\delta,\\
\tau^{2k+2}&= \inf\Big\{ t\geq \tau^{2k+1}: u_t\in E^{\delta^{-1}}_{\delta} \Big\}\wedge\rho_T^\delta.
%\bar{\gamma}(u_t) \leq \frac{\delta}{2} \text{ or }\mathcal{G}\Big(\frac{\delta}{2},u_t\Big) = 0\Big\}.
\end{align*}
Note that we are hiding the dependence of $\tau^n$ on $\delta$ and $T$ for notational sake (to avoid too many subscripts). For $t\in[\tau^{2k},\tau^{2k+1}]$, define $\eta^{\delta}_t = \beta^*_t$ (where $(\beta^*_t)_{t\geq0}$ is as in the theorem). On the other hand for $t\in(\tau^{2k+1},\tau^{2k+2})$, define $\eta^{\delta}_t = \eta^\delta_{\tau^{2k+1}}$, and for $t\in [0,\tau^0)$ define $\eta^\delta_t  = 0$. If $t\in [\rho^\delta_T, T]$, then define $\eta^{\delta}_t = \eta^{\delta}_{\rho^{\delta}_T}$. Let
\begin{equation}
\mathcal{D}_{\delta,T} := \bigcup_{k=0}^\infty(\tau^{2k+1}, \tau^{2k+2}). \label{eqn D delta t definition}
\end{equation}
\vspace{0.3cm}
\noindent The following lemma shows that the process $(\eta^\delta_t)_{t\in[0, T]}$ is well-defined for any $T>0$.  

\begin{lem}
Let $T,\delta > 0$. Then there exists $n\geq1$ such that $\tau^n = \rho^\delta_T$ almost surely. In particular the process $(\eta^\delta_t)_{t\in[0, T]}$ described above is well-defined for any $T>0, \delta>0$.
\end{lem}

\begin{proof}
Suppose for a contradiction that $\tau^n <\rho^\delta_T$ for all $n\geq1$. 

\vspace{0.3cm}
\noindent\textit{Step 1:} We first claim that for all $\kappa >0$, there exists $k_0\geq1$ such that for all $k\geq k_0$ $|\Gamma(u_{\tau^{2k}}) - \Gamma(u_{\tau^{2k+1}})| \leq \kappa$ almost surely.

To see this, suppose otherwise.  Then for some $\kappa >0$ there exists a subsequence $(k_r)_{r\geq1}$ such that $| \Gamma(u_{\tau^{2k_r}}) - \Gamma(u_{\tau^{2k_r+1}}) | \geq \kappa$ for all $r\geq1$.  Now it is clear by continuity of $(u_t)_{t\geq0}$ and Lemma \ref{Lemma E delta nested} below that $u_{\tau^n} \in E^{2\delta^{-1}}_{\delta/3}$ for all $n\geq1$.  Thus by Lemma \ref{Lemma Uniform Continuity Gamma}, there exists $\upsilon > 0$ such that for all $r$, $\|u_{\tau^{2k_r}} - u_{\tau^{2k_r+1}}\| \geq \upsilon$, which contradicts the continuity of $(u_t)_{t\geq0}$. 

By the claim we thus have that for $k$ sufficiently large
\begin{equation}\label{eqn Interval Inclusion}
\left[\Gamma(u_{\tau^{2k+1}})-\delta / 2, \Gamma(u_{\tau^{2k+1}})+\delta / 2\right] \subset \left[\Gamma(u_{\tau^{2k}})-\delta, \Gamma(u_{\tau^{2k}})+\delta\right]. 
\end{equation}

\vspace{0.3cm}
\noindent\textit{Step 2:}  The second step is to establish that there exists a constant $\kappa^* > 0$
\begin{equation}
\label{claim: inf osc}
\limsup_{k\to\infty} \norm{u_{\tau^{2k+1}}-u_{\tau^{2k}}} \geq \kappa^*.
\end{equation}
This implies that there are infinitely many nontrivial oscillations over the interval $[0,\rho_T^\delta]$, which is clearly a contradiction and thus proves the lemma.

The rest of the proof is thus devoted to showing \eqref{claim: inf osc}.
There are two (non exclusive) possible reasons why $u_{\tau^{2k+1}} \notin E_{\delta / 2}^{\delta^{-1}}$. The first possibility is that there exists an $\alpha \in \left[\Gamma(u_{\tau^{2k+1}})-\delta / 2, \Gamma(u_{\tau^{2k+1}})+\delta / 2\right]$ such that $\gamma(\alpha,u_{\tau^{2k+1}} ) \leq  \delta / 2$. In this case, since $u_{\tau^{2k}}\in E_{\delta}^{\delta^{-1}}$ by definition, thanks to \eqref{eqn Interval Inclusion} it must be that $\gamma(\alpha,u_{\tau^{2k}}) > \delta$. This means that 
\begin{align*}
\norm{u_{\tau^{2k+1}}-u_{\tau^{2k}}} = \norm{v^0_{\tau^{2k+1}} - v^0_{\tau^{2k}}} &\geq \frac{| \langle v^0_{\tau^{2k+1}} - v^0_{\tau^{2k}},\varphi''_\alpha\rangle |}{\norm{\varphi''_0}} \\
&= \frac{\gamma(\alpha,u_{\tau^{2k}}) - \gamma(\alpha,u_{\tau^{2k+1}})}{\norm{\varphi''_0}}\geq \frac{\delta}{2\norm{\varphi''_0}},
\end{align*}
so that \eqref{claim: inf osc} holds in this case.

The other possible reason why $u_{\tau^{2k+1}} \notin E_{\delta / 2}^{\delta^{-1}}$ is that there exists an $\alpha$ such that $|{\alpha}-\Gamma(u_{\tau^{2n+1}})| \geq \frac{\delta}{2}$ and 
\[
\norm{u_{\tau^{2k+1}} - \varphi_{\Gamma(u_{\tau^{2k+1}})}} \geq \norm{u_{\tau^{2k+1}} - \varphi_{{\alpha}}} -  \frac{1}{8}\delta^3(\delta\norm{\varphi_0'}/2+2\delta^{-1})^{-1}. 
\]
Now by Taylor's theorem, for some $\lambda \in [0,1]$,
\begin{align*}
\norm{\varphi_{\Gamma(u_{\tau^{2k+1}})} - \varphi_{\Gamma(u_{\tau^{2k}})}} &\leq \norm{\varphi'_{\lambda\Gamma(u_{\tau^{2k+1}}) + (1-\lambda)\Gamma(u_{\tau^{2k}})}}|\Gamma(u_{\tau^{2k+1}})-\Gamma(u_{\tau^{2k}})| \\ 
&= \norm{\varphi_0'}|\Gamma(u_{\tau^{2k+1}})-\Gamma(u_{\tau^{2k}})| \leq\norm{\varphi_0'}\kappa,
\end{align*}
for any $\kappa>0$ by taking $k$ large enough by Step 1.
From Lemma \ref{Lemma E delta temporary} and the definition of $E_\delta$,
\begin{multline*}
\norm{u_{\tau^{2k}} - \varphi_{\Gamma(u_{\tau^{2k}})}} \leq \norm{u_{\tau^{2k}} - \varphi_{{\alpha}}} - \min\left(\delta^2,({\alpha}-\Gamma(u_{\tau^{2k}}))^2\right)\times \frac{\delta}{\delta\norm{\varphi_0'} + 2\delta^{-1}}. 
\end{multline*}
Thus by the reverse triangle inequality, we find using the above three equations that
\begin{multline}
\norm{u_{\tau^{2k+1}} - u_{\tau^{2k}}} \geq \norm{u_{\tau^{2k+1}} - \varphi_{\Gamma(u_{\tau^{2k+1}})}} - \norm{u_{\tau^{2k}} - \varphi_{\Gamma(u_{\tau^{2k}})}} - \norm{\varphi_{\Gamma(u^{2k+1})} - \varphi_{\Gamma(u^{2k})}} \\
\geq  \norm{u_{\tau^{2k+1}} - \varphi_{{\alpha}}} - \norm{u_{\tau^{2k}} - \varphi_{{\alpha}}}+ \min\left(\delta^2,({\alpha}-\Gamma(u_{\tau^{2k}}))^2\right)\times \frac{\delta}{\delta\norm{\varphi_0'} + 2\delta^{-1}}\\ -  \frac{\delta^3}{8(\delta\norm{\varphi_0'} / 2+ 2\delta^{-1})} - \norm{\varphi_0'}\kappa,\label{eqn temp kappa}
\end{multline}
for any $\kappa>0$ by taking $k$ large enough.
Moreover, for such $k$
\begin{align*}
\left|{\alpha}-\Gamma(u_{\tau^{2k}})\right| &\geq \left|{\alpha}-\Gamma(u_{\tau^{2k+1}})\right| - \left|\Gamma(u_{\tau^{2k+1}})-\Gamma(u_{\tau^{2k}})\right| \\
&\geq  \left|{\alpha}-\Gamma(u_{\tau^{2k+1}})\right| - \kappa \geq \frac{\delta}{2} - \kappa.
\end{align*}
Therefore, for $\kappa \in(0, \delta/2)$ and $k$ large enough
\begin{align*}
& \min\left(\delta^2,(\bar{\alpha}-\Gamma(u_{\tau^{2k}}))^2\right)\times \frac{\delta}{\delta\norm{\varphi_0'} + 2\delta^{-1}} - \frac{\delta^3}{8(\delta\norm{\varphi_0'} / 2+ 2\delta^{-1})} -\norm{\varphi_0'}\kappa \\
 &\geq \left(\frac{\delta}{2} - \kappa\right)^2\times \frac{\delta}{\delta\norm{\varphi_0'} + 2\delta^{-1}} - \frac{\delta^3}{4(\delta\norm{\varphi_0'} + 4\delta^{-1})}  -\norm{\varphi_0'}\kappa\\
  &\geq \frac{\delta^{2}}{2(\delta\|\varphi'_0\| + 4\delta^{-1})^2} -\kappa\left(\frac{\delta^2}{\delta\norm{\varphi_0'} + 2\delta^{-1}}+ \norm{\varphi_0'}\right) \geq \frac{\delta^2}{4(\delta\|\varphi_0\| + 4\delta^{-1})^2} =: \bar{\kappa},
\end{align*}
by choosing $\kappa$ small enough, and then $k$ large enough.
Applying this to \eqref{eqn temp kappa}, we then have that for $k$ large enough, $\norm{u_{\tau^{2k+1}} - u_{\tau^{2k}}} \geq \bar{\kappa}+ \norm{u_{\tau^{2k+1}} - \varphi_{\alpha}} - \norm{u_{\tau^{2k}} - \varphi_{\alpha}}$. Since $\norm{u_{\tau^{2k+1}} - \varphi_{\alpha}} - \norm{u_{\tau^{2k}} - \varphi_{\alpha}} \geq - \norm{u_{\tau^{2k+1}} - u_{\tau^{2k}}}$, we find that $\norm{u_{\tau^{2k+1}} - u_{\tau^{2k}}}  \geq \frac{1}{2}\bar{\kappa}$, for $k$ large enough.  This shows that \eqref{claim: inf osc} holds in this second case too, which proves the result.
\end{proof}

We can now turn to the proof of Theorem \ref{Proposition Nonlinear Inequality}.

\begin{proof}[Proof of Theorem \ref{Proposition Nonlinear Inequality}]
We first prove the theorem for the case $s=0$ and $t = T$. We also assume for now that there exists a $\delta$ such that $u_0 \in E_\delta$, and so $\tau^0 = 0$. Assume that $\delta$ is small enough so that $\sup_{r\in [0,T]}\norm{v^0_t} \leq \delta^{-1}$, so that $\rho^\delta_T = T$. Define
\begin{equation}
\label{y}
y^\delta_r = u_r - \varphi_{\eta^\delta_r}, \qquad r\in[0, T],
\end{equation}
where $(\eta_r^\delta)_{r\in[0,T]}$ is defined above. The process $(\eta^\delta_r)_{r\in[0, T]}$ has been constructed piecewise on each interval $[\tau^{2k}, \tau^{2k+2})$ for $k\geq0$ (with $\tau^0:=0$), with $\eta^\delta$ satisfying the SDE \eqref{beta} on $[\tau^{2k}, \tau^{2k+1}]$ and being constant on the interval $[\tau^{2k+1}, \tau^{2k+2})$ (and equal to $\beta^*_{\tau^{2k+1}}$). Then
\begin{equation*}
\norm{y^\delta_T}^2 - \norm{y^\delta_0}^2 \leq \sum_{k=0}^\infty \left(\norm{y^\delta_{\tau^{2k+1}}}^2 - \norm{y^\delta_{\tau^{2k}}}^2 + \norm{y^\delta_{\tau^{2k+2}}}^2 - \norm{y^\delta_{\tau^{2k+1}}}^2\right).
\end{equation*}

By definition, on $\mathcal{D}_{\delta, T}^c:= [0,T]\backslash\mathcal{D}_{\delta, T}$, the process $\eta^\delta$ follows the solution of the SDE \eqref{beta}.  We may therefore apply Theorem \ref{thm: ineq better} to see that
\begin{multline*}
\norm{y^\delta_T}^2 - \norm{y^\delta_0}^2 \leq \int_{\mathcal{D}_{\delta, T}^c}\Xi(y^\delta_r)+ 2\langle f(\varphi_{\eta^\delta_r} + y^{\delta}_r) - f(\varphi_{\eta^\delta_r}), y^\delta_r\rangle dr +2\varepsilon \int_{\mathcal{D}_{\delta, T}^c}\langle y^\delta_r,B(r)dW^Q_r\rangle \\
+\varepsilon^2\int_{\mathcal{D}_{\delta, T}^c}\left[\mathrm{Tr}(Q) - \frac{\|Q^\frac{1}{2}\varphi'_{\eta_r^\delta}\|^2}{\gamma(\eta_r^\delta, v_r^0)}\right] dr+ \sum_{k=0}^\infty \left(\norm{y^\delta_{\tau^{2k+2}}}^2 - \norm{y^\delta_{\tau^{2k+1}}}^2\right).
\end{multline*}
Moreover, for $r\in\mathcal{D}_{\delta, T}^c$, $\eta^\delta_r = \beta^*_r$ by defintion.  Therefore 
\begin{multline*}
\norm{z^*_T}^2 - \norm{z^*_0}^2 \leq \int_{\mathcal{D}_{\delta, T}^c}\Xi(z^*_r)+ 2\langle f(\varphi_{ \beta^*_r} + z^*_r) - f(\varphi_{ \beta^*_r}), z^*_r\rangle dr \\ +2\varepsilon \int_{\mathcal{D}_{\delta, T}^c}\langle z^*_r,B(r)dW^Q_r\rangle+\varepsilon^2\int_{\mathcal{D}_{\delta, T}^c}\left[\mathrm{Tr}(Q) - \frac{\|Q^\frac{1}{2}\varphi'_{ \beta^*_r}\|^2}{\gamma( \beta^*_r, v_r^0)}\right] dr + R(\delta)
\end{multline*}
where 
\[
R(\delta):= \sum_{k=0}^\infty \left(\norm{y^\delta_{\tau^{2k+2}}}^2 - \norm{y^\delta_{\tau^{2k+1}}}^2\right).
\]
Thanks to Lemma  \ref{Lemma D delta}, it can thus be seen that the following claim is sufficient to establish the theorem.
\vspace{0.3cm}

\noindent\textit{Claim:} We claim that $R(\delta) \to 0$ almost surely as $\delta\to0$.

\vspace{0.3cm}
\noindent\textit{Proof of claim:} 
By definition of the process $(\eta_t^\delta)_{t\in[0, T]}$ we have that for any $k\geq0$, 
\[
\|y^\delta_{\tau^{2k+2}}\|^2 \leq \|y^\delta_{\tau^{2k+2}-}\|^2.
\]
Therefore
\begin{equation*}
R(\delta)\leq \sum_{k=0}^\infty \left(\norm{y^\delta_{t\wedge \tau^{2k+2}-}}^2 - \norm{y^\delta_{t\wedge\tau^{2k+1} }}^2\right).
\end{equation*}
Now since $\eta^\delta$ is constant on $[\tau^{2k+1}, \tau^{2k+2})$, and using the fact that $y_r^\delta = v_r^0 + \varphi_0 - \varphi_{\eta^\delta_{\tau^{2k+1}}}$ for all $r\in [\tau^{2k+1}, \tau^{2k+2})$,
\[
\norm{y^\delta_{ \tau^{2k+2}-}}^2 - \norm{y^\delta_{\tau^{2k+1} }}^2 = \|v^0_{ \tau^{2k+2}}\|^2 - \|v^0_{ \tau^{2k+1}}\|^2 + 2\langle v^0_{ \tau^{2k+2}} - v^0_{ \tau^{2k+1}}, \varphi_0 - \varphi_{\eta^\delta_{\tau^{2k+1}}}\rangle.
\]
Moreover, by It\^o's lemma and Proposition \ref{K-S} we can see that for any $0\leq s\leq r$
\begin{multline*}
\norm{v^0_{r}}^2 = \norm{P^A_{r- s}v^0_{s}}^2+2\int_{s}^{r}\langle P^A_{r - \theta}v^0_\theta,P^A_{r-\theta}(f(\varphi_{0} + v^0_\theta) - f(\varphi_{0}))\rangle d\theta \\ 
+  2\varepsilon \int_{s}^{r}\langle P^A_{r - \theta}v^0_\theta, P^A_{r - \theta}B(\theta)dW^Q_\theta\rangle + \varepsilon^2\int_{s}^{r}\mathrm{Tr}\left( B^*(\theta)(P^A_{r- \theta})^*QP^A_{r-\theta}B(\theta)\right)d\theta. 
\end{multline*}
For $r\in \mathcal{D}_{\delta, T}$, write $\bar{r} = \inf \lbrace \tau^{2k} : k\geq0,\ \tau^{2k} \geq r\rbrace$ and $\underline{r} = \sup\lbrace \tau^{2k+1} : k\geq0,\ \tau^{2k+1} \leq r\rbrace$. We thus see that 
\begin{multline}
|R(\delta)| \leq \sum_{k=0}^\infty\left| \|v^0_{\tau^{2k+2}}\|^2 - \|v^0_{ \tau^{2k+1}}\|^2 + 2\langle v^0_{\tau^{2k+2}} - v^0_{\tau^{2k+1}}, \varphi_0 - \varphi_{\eta^\delta_{ \tau^{2k+1}}}\rangle\right| \\
\leq \sum_{k=0}^\infty  \Big|\norm{P^A_{\tau^{2k+2}-   \tau^{2k+1}}v^0_{ \tau^{2k+1}}}^2 - \|v^0_{ \tau^{2k+1}}\|^2\Big| \\ 
+ 2\int_{\mathcal{D}_{\delta, T}}\Big|\Big\langle P^A_{\bar{r}-r}v^0_r, P^A_{\bar{r}-r}[f(\varphi_{0} + v^0_r) - f(\varphi_{0})]\Big\rangle\Big| dr + 2\varepsilon \Big|\int_{\mathcal{D}_{\delta, T}}\Big\langle P^A_{\bar{r}-r}v^0_r, P^A_{\bar{r}-r}B(r)dW^Q_r\Big\rangle\Big|\\ 
+ \varepsilon^2 \int_{\mathcal{D}_{\delta, T}}\mathrm{Tr}\left( B^*(r)(P^A_{\bar{r}-r})^*QP^A_{\bar{r}-r}B(r)\right)dr \\
+ 2\sum_{k=0}^\infty\Big|\Big\langle P^A_{\tau^{2k+2} - \tau^{2k+1}}v^0_{\tau^{2k+1}} - v^0_{ \tau^{2k+1}}, \varphi_0 - \varphi_{\eta^\delta_{\tau^{2k+1}}}\Big\rangle\Big|\\
+ 2 \int_{\mathcal{D}_{\delta, T}}\Big|\Big\langle P^A_{\bar{r} - r}[f(\varphi_{0} + v^0_r) - f(\varphi_{0})], \varphi_0 - \varphi_{\eta^\delta_{ \underline{r}}}\Big\rangle\Big| dr
+2\varepsilon \Big|\int_{\mathcal{D}_{\delta, T}}\Big\langle \varphi_0 - \varphi_{\eta^\delta_{ \underline{r}}}, P^A_{\bar{r}-r}B(r)dW^Q_r\Big\rangle\Big|.\label{eq temporary v bound}
\end{multline}
By Lemma \ref{Lemma D delta}, we have that $|\mathcal{D}_{\delta, T}|\to0$ almost surely as $\delta\to0$.  Therefore ${\tau}^{2k+2}-   {\tau}^{2k+1} \to0$ almost surely, for every $k\geq0$. Since $(P^A_r)_{r\geq0}$ is a strongly continuous semigroup (see Assumption \ref{assump A} (i)), we thus have that first term on the right-hand side of \eqref{eq temporary v bound} converges to 0 almost surely.  Moreover, so does the fifth term thanks to the dominated convergence theorem.  All the other non-stochastic integral terms are similarly easy to handle thanks the dominated convergence theorem.  The stochastic integral term can be also shown to converge to $0$ almost surely as $\delta\to0$ by taking expectations. We note also that the other claim in the theorem -- i.e. the almost sure uniqueness of $\beta^*_t$ for each $t$ -- is proved in the course of Lemma \ref{Lemma D delta}.

We assumed at the start of the proof that $s=0$, $t=T$ and $u_0\in E_\delta$ for some sufficiently small $\delta$. We now treat the more general case when these assumptions do not hold. First, if $s\neq 0$, then almost surely there exists a $\delta > 0$ such that $u_s \in E_\delta$ (this is noted in the proof of Lemma \ref{Lemma D delta}). The proof of this case now proceeds exactly as above, with $\tau^0$ redefined to be $s$. Second, suppose that $s=0$ but $u_0 \notin E_\delta$ for any $\delta > 0$. Since $\left|\mathcal{D}_{\delta,T}\right| \to 0$, it follows that $\tau^0 \to 0$ as $\delta \to 0$, and the result still holds.
\end{proof}

\subsection{Auxiliary Lemmas}
\label{auxiliary lemmas}

\begin{lemma}\label{Lemma E delta temporary}
Suppose that $u\in E_\delta^M$. Then for all $\alpha\in [\Gamma(u)-\delta,\Gamma(u)+\delta]$, $\alpha \neq \Gamma(u)$,
\begin{equation*}
\norm{u-\varphi_\alpha} - \norm{u-\varphi_{\Gamma(u)}} >  \frac{ \delta(\alpha-\Gamma(u))^2}{\delta\norm{\varphi_0'}+ 2M}.
\end{equation*}
\end{lemma}
\begin{proof}
By Taylor's theorem, for all $\alpha\in [\Gamma(u)-\delta,\Gamma(u)+\delta]$, $\alpha\neq\Gamma(u)$, for some $\bar{\alpha}$ between $\alpha$ and $\Gamma(u)$
\begin{equation*}
\norm{u-\varphi_\alpha}^2 = \norm{u-\varphi_{\Gamma(u)}}^2 +\gamma(\bar{\alpha},u)(\alpha-\Gamma(u))^2 > \norm{u-\varphi_{\Gamma(u)}}^2 + \delta (\alpha-\Gamma(u))^2,
\end{equation*}
the last inequality following from the definition of $E_\delta^M$. Hence
\begin{equation*}
(\norm{u-\varphi_\alpha} - \norm{u-\varphi_{\Gamma(u)}}) (\norm{u-\varphi_\alpha} + \norm{u-\varphi_{\Gamma(u)}}) > \delta (\alpha-\Gamma(u))^2,
\end{equation*}
so that
\begin{align*}
\norm{u-\varphi_\alpha} - \norm{u-\varphi_{\Gamma(u)}} >& \frac{ \delta(\alpha-\Gamma(u))^2}{\norm{u-\varphi_\alpha} + \norm{u-\varphi_{\Gamma(u)}}} 
\geq\frac{ \delta(\alpha-\Gamma(u))^2}{\norm{\varphi_{\Gamma(u)}-\varphi_\alpha} + 2\norm{u-\varphi_{\Gamma(u)}}}.
\end{align*}
Now again by Taylor's theorem, for some $\lambda \in [0,1]$, it holds that $\varphi_\alpha - \varphi_{\Gamma(u)} = ((1-\lambda)\varphi_\alpha' + \lambda\varphi_{\Gamma(u)}')(\alpha-\Gamma(u))$. Thus
\[
\norm{\varphi_\alpha - \varphi_{\Gamma(u)}}\leq |\alpha-\Gamma(u)|\norm{\varphi_0'}.
\]
Therefore, making use of the definition of $E_\delta^M$,
\begin{align*}
\norm{u-\varphi_\alpha} - \norm{u-\varphi_{\Gamma(u)}} > & \frac{ \delta(\alpha-\Gamma(u))^2}{ |\alpha-\Gamma(u)|\norm{\varphi_0'}+ 2\norm{u-\varphi_{\Gamma(u)}}}\nonumber\\
\geq & \frac{ \delta(\alpha-\Gamma(u))^2}{ |\alpha-\Gamma(u)|\norm{\varphi_0'}+ 2M}
\geq  \frac{ \delta(\alpha-\Gamma(u))^2}{\delta\norm{\varphi_0'}+ 2M}.
\end{align*}
\end{proof}

\begin{lemma}\label{Lemma E delta nested}
If $\delta_1 < \delta_2$, then $E^M_{\delta_2}\subset E^M_{\delta_1}$.
\end{lemma}

\begin{proof} 
For $u\in E^M_{\delta_2}$, to show that $u\in E^M_{\delta_1}$ the only thing that is slightly difficult to check is that 
for all $\alpha\in\R\backslash[\Gamma(u)-\delta_1,\Gamma(u)+\delta_1]$, $\norm{u - \varphi_{\alpha}}> \norm{u - \phi_{\Gamma(u)}}  +  \delta_1^3(\delta_1\norm{\varphi_0'}+ 2M)^{-1}$.
For $\alpha\in\R\backslash[\Gamma(u)-\delta_2,\Gamma(u)+\delta_2]$ this follows from the fact that $u\in E^M_{\delta_2}$ and $\delta_1 < \delta_2$.  If $\alpha\in[\Gamma(u)-\delta_2,\Gamma(u)+\delta_2]$ but $\alpha\not\in[\Gamma(u)-\delta_1,\Gamma(u)+\delta_1]$ then by Lemma \ref{Lemma E delta temporary} 
\[
\norm{u - \varphi_{\alpha}} - \norm{u - \phi_{\Gamma(u)}}  >  \delta_2\delta_1^2(\delta_2\norm{\varphi_0'}+ 2M)^{-1} >\delta_1^3(\delta_1\norm{\varphi_0'}+ 2M)^{-1}. 
\]
\end{proof}

\begin{lemma}\label{Lemma Uniform Continuity Gamma}
For all $\theta \in (0, \delta)$, there exists $\zeta>0$ such that for all $u,w\in E_\delta^M$, if $\norm{u-w} \leq \zeta$ then $|\Gamma(u)-\Gamma(w)| \leq \theta$.
\end{lemma}
\begin{proof}
Let  $\theta \in (0, \delta)$ and define $\zeta = \frac{1}{2}\theta^2\delta(\delta\norm{\varphi_0'}+ 2M)^{-1}$. Assume that $\norm{u-w}\leq \zeta$. Now
\begin{equation*}
\norm{w-\varphi_{\Gamma(u)}} \leq \norm{u-\varphi_{\Gamma(u)}} + \norm{w-u}
 \leq \norm{u-\varphi_{\Gamma(u)}}+ \zeta.
\end{equation*}
Thus the Lemma will follow if we can establish the following claim.

\vspace{0.3cm}
\noindent\textit{Claim:} $\inf_{\alpha \notin [\Gamma(u)-\theta,\Gamma(u)+\theta]}\norm{w-\varphi_\alpha} >  \norm{u-\varphi_{\Gamma(u)}} + \zeta$.

\vspace{0.3cm}
\noindent Indeed, if this claim is true we then have that
\[
\norm{w-\varphi_{\Gamma(u)}} <\inf_{\alpha \notin [\Gamma(u)-\theta,\Gamma(u)+\theta]}\norm{w-\varphi_\alpha}
\]
so that $\Gamma(w)$ is certainly within a distance $\theta$ of $\Gamma(u)$.
To prove the claim, from the reverse triangle inequality, 
\[
\inf_{\alpha \notin [\Gamma(u)-\theta,\Gamma(u)+\theta]}\norm{w-\varphi_\alpha} \geq \inf_{\alpha \notin [\Gamma(u)-\theta,\Gamma(u)+\theta]}\norm{u-\varphi_\alpha} - \norm{w-u}.
\]
Suppose that $\inf_{\alpha \notin [\Gamma(u)-\theta,\Gamma(u)+\theta]}\norm{u-\varphi_\alpha} = \norm{u-\varphi_{\tilde{\alpha}}}$. Then $\tilde{\alpha}\in\R\backslash (\Gamma(u)-\theta,\Gamma(u)+\theta)$.  We claim that $\norm{u-\varphi_{\tilde{\alpha}}} > \norm{u-\varphi_{\Gamma(u)}} + \theta^2\delta(\delta\norm{\varphi_0'}+ 2M)^{-1}$. To see this, if $\tilde{\alpha}\not\in [\Gamma(u)-\delta,\Gamma(u)+\delta]$, then by the definition of $E_\delta^M$, 
\[
\norm{u-\varphi_{\tilde{\alpha}}}> \norm{u-\varphi_{\Gamma(u)}} + \delta^3(\delta\norm{\varphi_0'}+ 2M)^{-1}>  \norm{u-\varphi_{\Gamma(u)}} + \theta^2\delta(\delta\norm{\varphi_0'}+ 2M)^{-1}.
\] 
On the other hand if $\tilde{\alpha}\in [\Gamma(u)-\delta,\Gamma(u)+\delta]$ (recall that $\tilde{\alpha}\in\R\backslash (\Gamma(u)-\theta,\Gamma(u)+\theta)$), by Lemma \ref{Lemma E delta temporary} we have that 
\begin{align*}
\norm{u-\varphi_{\tilde{\alpha}}} > \norm{u-\varphi_{\Gamma(u)}} + \theta^2\delta(\delta\norm{\varphi_0'}+ 2M)^{-1}.
\end{align*}
\end{proof}

%\begin{lemma}
%$E^M_\delta$ is an open subset of $E$.
%\end{lemma}
%\begin{proof}
%\ing{This is a consequence of the auxiliary Lemma \ref{Lemma Uniform Continuity Gamma} given below -- is this obvious?  Need to check that if $(u_n)_{n\geq1}$ is a family such that $\|u - u_n\|\to 0$ then $u\in E_\delta$ i.e. i) $u\in E_\delta$, ii) $\|u - \varphi_0\|$ and iii) of the definition.  I agree it should be true, but maybe we need to include a few more details...but maybe we don't even need this Lemma?!}
%\end{proof}

\begin{lemma}\label{Lemma E bar}
Let $\bar{E}:= \bigcup_{\delta\in(0,1)}E^{\delta^{-1}}_{\delta}$.  Then $\bar{E}$ is the set of all $u \in E$ such that $\alpha\mapsto \norm{u - \varphi_\alpha}$ has a unique global minimum $\Gamma(u)$ such that $\gamma(\Gamma(u),u) > 0$.
\end{lemma}
\begin{proof}
Suppose that $u \in E$ is such that $\alpha\mapsto \norm{u - \varphi_\alpha}$ has a unique global minimum $\Gamma(u)$ and $\gamma(\Gamma(u),u) > 0$. We prove that there exists a $\bar{\delta}$ such that $u\in E^{\bar{\delta}^{-1}}_{\bar{\delta}}$. It follows from the continuity of $\gamma$ that if $\gamma(\Gamma(u),u) > 0$, then there exists some $\delta > 0$ such that $\gamma(\alpha,u) > \delta$ for all $\alpha$ in some neighborhood $[\Gamma(u)-\delta,\Gamma(u)+\delta]$ of $\Gamma(u)$. 

Suppose for a contradiction that there is a sequence of points $\alpha_j \notin [\Gamma(u)-\delta,\Gamma(u)+\delta]$ such that $\norm{u - \varphi_{\alpha_j}} \to \norm{u - \varphi_{\Gamma(u)}}$ as $j\to \infty$. By Lemma \ref{beta existence}, there must exist a compact set $K$ such that $\alpha_j \in K$ for all $j$. Therefore there must exist a $\xi \in K$ such that for a subsequence $p_j$, $\alpha_{p_j} \to \xi$. By continuity, $\norm{u-\varphi_\xi} = \norm{u -\varphi_{\Gamma(u)}}$. This contradicts the uniqueness of the global minimum of $u$. Therefore there must exist a $\kappa$ such that for all $\alpha \notin [\Gamma(u)-\delta,\Gamma(u)+\delta]$, $\norm{u - \varphi_\alpha} > \norm{u - \varphi_{\Gamma(u)}} + \kappa$. Let $\delta^*$ be such $(\delta^*)^3 / (\norm{\varphi_0'}\delta^* + 2M) < \kappa$.

Let $\bar{\delta} \leq \rm{min}\left( \delta,\delta^*\right)$. It may be seen that $u \in E^{\bar{\delta}^{-1}}_{\bar{\delta}}$, if $\bar{\delta}^{-1}\geq \norm{u-\varphi_0}$.
\end{proof}

\begin{lemma}\label{Lemma D delta}
Under Assumptions \ref{Assmpt Long Time 1} and \ref{Assmpt Long Time 2}, $|\mathcal{D}_{\delta, T} |$ tends to $0$ as $\delta\to 0$ almost surely.
\end{lemma}
\begin{proof}%[Proof of Lemma \ref{Lemma D delta}]
It suffices for us to prove that $\int_{\Omega} \int_{[0,T]}\one(u_t \notin E^{\delta^{-1}}_{\delta})dt d\mathbb{P} \to 0$  as $\delta \to 0$.
By Fubini's theorem,
\begin{equation*}
\int_{\Omega} \int_{[0,T]}\one(u_t \notin E^{\delta^{-1}}_{\delta})dt d\mathbb{P} =\int_{[0,T]} \int_{\Omega} \one(u_t \notin E^{\delta^{-1}}_{\delta}) d\mathbb{P}dt.
\end{equation*}
It thus suffices for us to show that for Lebesgue almost every $t\in[0,T]$,
$\mathbb{P}(u_t \notin E^{\delta^{-1}}_{\delta})\to 0$, as $\delta \to 0$.  Thanks to the inclusion relation of Lemma \ref{Lemma E delta nested}, this will follow if we can show that $\mathbb{P}(u_t \notin \bar{E}) = 0$ for almost every $t\in[0, T]$, where $\bar{E}:= \bigcup_{\delta\in(0,1)}E^{\delta^{-1}}_\delta$ is as in Lemma \ref{Lemma E bar}. Since $u_t = v^0_t + \varphi_0$, this is equivalent to showing that, for almost all $t\in[0,T]$,
\begin{equation}\label{eqn final ut not in E}
\mathbb{P}(v^0_t \notin \bar{E}-\varphi_0) = 0.
\end{equation}
We establish \eqref{eqn final ut not in E} using the Girsanov theorem. We recall the definition of the process $v^0 = (v^0_t)_{t\geq0}$ as the solution to 
$dv^0_t =\left(Av^0_t + f(\varphi_0 + v^0_t) - f(\varphi_0)\right)dt + B(t)dW_t$, and introduce the process $X=(X_t)_{t\geq0} \subset H$ as the solution to
\begin{align*}
dX_t :=& AX_tdt + B(t)dW_t,
\end{align*}
with $v^0_0 = X_0 $. 
Note that by the Lipschitz assumption on $f$
\begin{align*}
&\sup_{t\in [0,T]}\E\left[ \exp\left(\norm{ Q^{\frac{1}{2}} B(t)^{-1}(f(\varphi_0 + X_t) - f(\varphi_0))}\right)\right]  \leq \sup_{t\in [0,T]}\E\left[\exp\left( C\norm{X_t}\right)\right],
\end{align*}
for some constant $C$, using also Assumption \ref{assump noise} (ii). Since $(X_t)_{t\geq0}$ is a Gaussian process the right-hand side is finite. Thus the Girsanov theorem \cite[Theorem 10.18]{D-Z} applies. This means that the law of $v^0$ (which is a probability measure on $\mathcal{C}([0,T],H)$) is absolutely continuous with respect to the law of $X$.  Thus \eqref{eqn final ut not in E} will be satisfied if
\begin{equation}\label{eqn final Xt not in E}
\mathbb{P}(X_t+\varphi_0 \notin \bar{E}) = 0.
\end{equation}
%Note that $X_t = P^A_{t-s}a + \int_0^t P^A_{t-s}B(s)dW_s$. 

To show \eqref{eqn final Xt not in E}, by Lemma \ref{Lemma E bar} it suffices to show that i) $\alpha\mapsto \norm{X_t + \varphi_0 - \varphi_\alpha}^2$ has a unique global minimum $\bar{\alpha}$ almost surely, and ii) $\gamma(\bar{\alpha}, X_t + \varphi_0) >0$.

To show i), let
\begin{equation} \label{eqn Zalpha}
Y_{\alpha,t} := 2\left\langle X_t,\varphi_\alpha-\varphi_0\right\rangle, \quad \mathrm{and} \quad Z_{\alpha,t} := Y_{\alpha,t} - \norm{\varphi_\alpha-\varphi_0}^2.
\end{equation}
Observe that $\inf_{\alpha\in\R} \norm{X_t + \varphi_0 - \varphi_\alpha}^2 = \inf_{\alpha\in\R} \left(\norm{X_t}^2 - Z_{\alpha,t}\right) = \|X_t\|^2 - \sup_{\alpha\in\R} Z_{\alpha,t}$.
It may thus be seen that $\bar{\alpha}$ is the unique global minimum of  $\alpha \mapsto \norm{X_t + \varphi_0 - \varphi_\alpha}^2$ if and only  $Z_{\bar{\alpha},t} > Z_{\alpha,t}$ for all $\alpha\neq\bar{\alpha}$.

Since $X_t$ is Gaussian, $(Z_{\alpha,t})_{\alpha\in\R}$ is a continuous $\R$-indexed Gaussian process, for fixed $t\in[0, T]$. We have that
% The mean is clearly
%\[
%\bE\left[ Z_{\alpha,t}\right] = \langle P^A(t)a,\varphi_\alpha-\varphi_0\rangle - \norm{\varphi_0 - \varphi_\alpha}^2.
%\]
\[
\bE\left[ (Z_{\alpha,t} -Z_{\beta,t} - \bE[Z_{\alpha,t}-Z_{\beta,t}])^2\right] = 4\langle \varphi_\alpha-\varphi_\beta,Q_t(\varphi_\alpha-\varphi_\beta)\rangle, \quad \forall \alpha, \beta \in\R,
\]
where $Q_t$ is defined as in Assumption \ref{Assmpt Long Time 1}.
By this assumption, the above variance is nonzero for all $\alpha\neq \beta$ and $t>0$. Then by \cite[Lemma 2.6]{kim-pollard:93}, $\alpha\mapsto Z_{\alpha,t}$ has a unique supremum almost surely.

It remains for us to show ii).  It can be seen that this will hold if $Z_{\bar{\alpha},t}'' \neq 0$ (the derivative with respect to $\alpha$), almost surely. Since $\bar{\alpha}$ is the unique maximum and by assumption $(Z'_{\alpha,t},Z''_{\alpha,t},Z_{\alpha,t}''',Z_{\alpha,t}'''')$ all exist, if $Z_{\bar{\alpha},t}'' = 0$ then it must also be the case that $Z_{\bar{\alpha},t}' = Z_{\bar{\alpha},t}''' = 0$ (this may be seen by Taylor expanding $Z_{\alpha,t}$ about $\bar{\alpha}$). The result thus follows from Lemma \ref{Lemma Z alpha derivatives} below.
\end{proof}
%\begin{rem}
%The following lemma uses a simple covering by small balls to bound the hitting probability of the Gaussian process $Z_{\alpha,t}$. One could probably obtain a more powerful result by using some of the techniques in \cite{dalang-khoshnevisan,nualart-viens} (for example). One could also weaken Assumption \ref{Assmpt Long Time 2} through a more detailed analysis of the interaction of $\lambda_i^\varphi(\alpha)$ and $\tilde{\mathcal{G}}(\alpha)$ (defined in the proof of Lemma \ref{Lemma constant CM}). If for example two or more of $(\lambda_i^{\varphi}(\alpha))_i$ are zero, but $\inf_{i\in\lbrace1,2,3\rbrace} \left|\tilde{\mathcal{G}}_i(\alpha)\right| > 0$, then the lemma still holds. We have avoided doing this to keep the statement of the theorem elegant.
%\end{rem}
\begin{lemma}\label{Lemma Z alpha derivatives}
Under Assumption \ref{Assmpt Long Time 2}, for any $t\geq 0$, the probability that there exists an $\alpha\in \R$ such that
\begin{equation}\label{eqn Z derivatives zero}
Z_{\alpha,t}' = Z_{\alpha,t}'' = Z_{\alpha,t}''' = 0
\end{equation}
is zero, where $Z_{\alpha,t}$ is defined in \eqref{eqn Zalpha}.
\end{lemma}

\begin{proof}
Fix $M>0$. We will show that the probability that \eqref{eqn Z derivatives zero} holds for any $\alpha \in [-M,M]$ is zero. The lemma then follows directly from a covering argument.

For $n > 0$ and $j\in \{1, \dots, n\}$, let $\alpha^n_j = -M + 2M(j-1/2) / n$. Let $B^n_j$ be the interval $[\alpha^n_j - M / n, \alpha^n_j + M/n]$. Fix $m>0$. Using the result Lemma \ref{Lemma constant CM} below, we find that% Let $\mathcal{Y}_n$ be the event
%\[
%\sup_{\xi\in [-M,M]} \norm{X_t} \leq n.%\lbrace |Y_{\xi,t}'|,|Y_{\alpha,t}''|,|Y_{\xi,t}'''|,|Y_{\xi,t}''''|\rbrace \leq n.
%\]
\begin{align*}
&\mathbb{P}\left( Z_{\alpha,t}' = Z_{\alpha,t}'' = Z_{\alpha,t}''' = 0,  \text{for some }\alpha \in [-M,M]\right) \\ 
%\leq \sum_{j=1}^n \mathbb{P}\left( Z_{\alpha,t}' = Z_{\alpha,t}'' = Z_{\alpha,t}''' = 0,  \text{for some }\alpha \in B^n_j\right) \\
&\quad \leq \mathbb{P}(\norm{X}_t > m) + \sum_{j=1}^n \mathbb{P}\left( Z_{\alpha,t}' = Z_{\alpha,t}'' = Z_{\alpha,t}''' = 0  \text{ for some }\alpha \in B^n_j, \text{ and } \norm{X_t} \leq m\right)  \\
&\quad\leq \mathbb{P}(\norm{X}_t > m) + C_M(m+1)^2 n^{-1}.
\end{align*}
We obtain the result by taking $m,n\to \infty$, such that $\mathbb{P}(\norm{X}_t > m)\to 0$ and $C_M(m+1)^2 n^{-1}\to 0$.
\end{proof}

The following lemma uses variables defined in the proof of Lemma \ref{Lemma Z alpha derivatives}.

\begin{lemma}\label{Lemma constant CM}
Under Assumption \ref{Assmpt Long Time 2}, for each $t>0$ there exists a positive constant $C_M$ independent of $m$ and $n$ such that
\[
 \mathbb{P}\left( Z_{\alpha,t}' = Z_{\alpha,t}'' = Z_{\alpha,t}''' = 0  \text{ for some }\alpha \in B^n_j, \text{ and } \norm{X_t} \leq m\right)  \leq C_M(m+1)^2n^{-2}.
 \]
\end{lemma}

\begin{proof}
Let $\mathcal{G}(\alpha) := - \norm{\varphi_0 - \varphi_\alpha}^2$ for $\alpha\in\R$. Let $\alpha \in B^n_j$. By Taylor's theorem, for some $\tilde{\alpha} \in B^n_j$,
\[
Z_{\alpha,t}''' = Z_{\alpha^n_j,t}''' + (\alpha-\alpha^n_j)Z_{\tilde{\alpha},t}'''' = Y_{\alpha^n_j,t}''' + \mathcal{G}'''(\alpha^n_j) +(\alpha-\alpha^n_j)\left(Y_{\tilde{\alpha},t}'''' + \mathcal{G}''''(\tilde{\alpha})\right),
\]
where $Y_{\alpha, t}$ is defined in \eqref{eqn Zalpha}.
%$Z_{\alpha,t}''' := Y_{\alpha,t}''' + \mathcal{G}'''(\alpha) = 0$
Now by the Cauchy-Schwarz inequality, $\norm{Y_{\tilde{\alpha},t}''''} \leq 2\norm{\varphi_{\tilde{\alpha}}''''}\norm{X_t}$. By assumption, $ \left|\mathcal{G}''''(\alpha^n_j)\right|$ possesses a uniform upper bound. We thus find that if $Z_{\alpha,t}''' = 0$ for some $\alpha \in B^n_j$ and $\norm{X_t} \leq m$, then since $|\alpha-\alpha^n_j| \leq n^{-1}$,
\begin{equation}
\left| Y_{\alpha^n_j,t}''' + \mathcal{G}'''(\alpha^n_j)\right| \leq \mathcal{K}(m+1)n^{-1},\label{eqn Y bound 1}
\end{equation}
for some constant $\mathcal{K}$ which is independent of $n$, $j$ and $m$. We find similarly (readjusting the constant $\mathcal{K}$) that if $Z_{\alpha,t}' = Z_{\alpha,t}'' = 0$, then
\begin{align}
&\left| Y_{\alpha^n_j,t}' + \mathcal{G}'(\alpha^n_j)\right| \leq \mathcal{K}(m+1)n^{-1},\label{eqn Y bound 2}\\
&\left| Y_{\alpha^n_j,t}'' + \mathcal{G}''(\alpha^n_j)\right| \leq \mathcal{K}(m+1)n^{-1}.\label{eqn Y bound 3}
\end{align}
By construction, we can see that $4K^\varphi(\alpha^n_j, t)$ (as defined in Assumption \ref{Assmpt Long Time 2}) is the covariance matrix of the $\R^3$-valued Gaussian random variable $(Y_{\alpha^n_j,t}', Y_{\alpha^n_j,t}'', Y_{\alpha^n_j,t}''')$.
Moreover, recall the definition of $O^\varphi(\alpha,t)$ in \eqref{eqn O definition} and define $\R^3 \ni {\bf Y} := O^{\varphi}(\alpha^n_j,t)\cdot(Y'_{\alpha^n_j,t},Y''_{\alpha^n_j,t},Y'''_{\alpha^n_j,t})$ and $\R^3 \ni\boldsymbol{\mathcal{G}} := O^{\varphi}(\alpha^n_j,t)\cdot(\mathcal{G}'(\alpha^n_j),\mathcal{G}''(\alpha^n_j),\mathcal{G}'''(\alpha^n_j))$ (these are matrix-vector multiplications).  It may be observed from \eqref{eqn Y bound 1}-\eqref{eqn Y bound 3} that there exists a constant $\mathcal{K}_0$ (independent of $n$, $j$ and $m$) such that if $Z_{\alpha,t}'=Z_{\alpha,t}''=Z_{\alpha,t}''' = 0$ for some $\alpha\in B^n_j$, then
\begin{align*}
\norm{{\bf Y} + \boldsymbol{\mathcal{G}}}_{\infty} \leq \mathcal{K}_0(m+1)n^{-1}.
%\left| R_{\alpha^n_j,t}''' + \mathcal{G}'''(\alpha^n_j)\right| \leq \mathcal{K}(m+1)n^{-1},\label{eqn Y bound 1}
%&\left| R_{\alpha^n_j,t}' + \mathcal{G}'(\alpha^n_j)\right| \leq \mathcal{K}(m+1)n^{-1},\label{eqn Y bound 2}\\
%&\left| R_{\alpha^n_j,t}'' + \mathcal{G}''(\alpha^n_j)\right| \leq \mathcal{K}(m+1)n^{-1}.\label{eqn Y bound 3}
\end{align*}
Here $\norm{\cdot}_\infty$ is the supremum norm over $\R^3$. It thus suffices for us to show that there exists a constant $C_M$ such that
\begin{equation}\label{eqn bound norm Y}
\mathbb{P}\left(\norm{{\bf Y} + \boldsymbol{\mathcal{G}}}_{\infty} \leq \mathcal{K}_0 (m+1)n^{-1}\right) \leq C_M (m+1)^2 n^{-2}.
\end{equation}
Now the covariance matrix of ${\bf Y}$ is $4\Lambda^\varphi(\alpha^n_j,t)$ (defined in Assumption \ref{Assmpt Long Time 2}), which means that the three elements of ${\bf Y}$ are mutually independent Gaussian variables, with variances $(4\lambda^\varphi_i(\alpha^n_j,t))_{i\in\lbrace 1,2,3\rbrace}$. 

We claim that there must exist a constant $\kappa > 0$ such that for all $\alpha \in [-M,M]$, no more than one of $(\lambda^\varphi_i(\alpha, t))_{i\in \lbrace1,2,3\rbrace}$ are less than $\kappa$. To see this, assume for a contradiction that there exists a sequence $(\alpha^r)_{r\in\N} \subset [-M,M]$ such that at least two of $(\lambda^\varphi_i(\alpha^r,t))_{i\in \lbrace 1,2,3\rbrace}$ are less than $r^{-1}$. From this we must be able to obtain a subsequence $(\tilde{\alpha}^p)_{p\in\N}$ such that for some $(k,l)\in \lbrace 1,2,3\rbrace$, $\lambda^\varphi_k(\tilde{\alpha}^p,t),\lambda^\varphi_l(\tilde{\alpha}^p,t)\leq p^{-1}$. By the compactness of $[-M,M]$, there must exist a point $\bar{\alpha}$ such that a subsequence of $(\tilde{\alpha}^p)_{p\in\N}$ converges to $\bar{\alpha}$. By continuity, $\lambda^\varphi_k(\bar{\alpha},t) = \lambda^\varphi_l(\bar{\alpha},t) = 0$. This contradicts Assumption \ref{Assmpt Long Time 2}.

Now the probability of a 1-dimensional Gaussian variable of variance $\Sigma$ being in some interval of width $\delta$ is upper-bounded by $\delta (2\pi \Sigma)^{-1/2}$. Since the variances of at least two of the $(\lambda^\varphi_i(\alpha^n_j,t))_{i\in\lbrace 1,2,3\rbrace}$ are lower-bounded by $4\kappa$, we have that
\begin{equation*}
\mathbb{P}\left(\norm{{\bf Y} + \boldsymbol{\mathcal{G}}}_{\infty} \leq \mathcal{K}_0 (m+1)n^{-1}\right) \leq \left(2\mathcal{K}_0\frac{m+1}{n}\right)^2 (8\pi\kappa)^{-1}.
\end{equation*}
This gives us \eqref{eqn bound norm Y}. %\ing{It looks much better. One final point - can you check the constants in the above inequality?  I can't see where the factor of $2$ comes from in front of $\mathcal{K}_0$, and I would have thought it would be $2\times4\kappa = 8\kappa$, rather than $4\kappa$, though I am sure I am wrong!}
\end{proof}

\subsection*{Acknowledgments}
The authors would like to thank P. Bressloff for a helpful discussion which got them started on the problem. They also thank W. Stannat, E. Lang and J. Kr\"uger for the invitation and interesting exchanges in Berlin.

\bibliographystyle{siam}
\bibliography{../StochasticWaves}

\end{document}